\documentclass[12pt]{amsart}
\usepackage{amssymb}
\usepackage{amsmath}
\usepackage[normalem]{ulem}
\usepackage{xcolor}

\usepackage{full page}

\allowdisplaybreaks

\newtheorem{thm}{Theorem}[section]
\newtheorem{prop}[thm]{Proposition}
\newtheorem{cor}[thm]{Corollary}
\newtheorem{lem}[thm]{Lemma}
\newtheorem{rem}[thm]{Remark}
\newtheorem{defn}[thm]{Definition}

\newtheorem{conj}[thm]{Conjecture}

\numberwithin{equation}{section}

\newcommand{\C}{\mathbb{C}}
\newcommand{\N}{\mathbb{N}}
\newcommand{\Z}{\mathbb{Z}}
\newcommand{\res}{\mbox{\rm Res}}
\newcommand{\wt}{\mbox{\rm wt}\ }

\begin{document}

\title[Generators and relations for $A_n(V)$]{On generators and relations for higher level Zhu algebras and applications}

\author{Darlayne Addabbo}
\address{Department of Mathematics University of Arizona, Tucson, AZ  85721}
\email{addabbo@math.arizona.edu}
\author{Katrina Barron}
\address{Department of Mathematics, University of Notre Dame, Notre Dame, IN 46556}
\email{kbarron@nd.edu}


\date{March 17, 2023}

\thanks{D. Addabbo was supported by AMS-Simons Foundation Travel Grant; K. Barron was supported by Simons Foundation Collaboration Grant 282095.}


\begin{abstract}
We give some general results about the generators and relations for the higher level Zhu algebras for a vertex operator algebra.  In particular, for any element $u$ in a vertex operator algebra $V$, such that $u$ has weight greater than or equal to $-n$ for $n \in \mathbb{N}$, we prove a recursion relation in the $n$th level Zhu algebra $A_n(V)$ and give a closed formula for this relation.  We use this and other properties of $A_n(V)$ to reduce the modes of $u$ that appear in the generators for $A_n(V)$ as long as $u \in V$ has certain properties (properties that apply, for instance, to the conformal vector for any vertex operator algebra or if $u$ generates a Heisenberg vertex subalgebra), and we then prove further relations in $A_n(V)$ involving such an element $u$.   We present general techniques that can be applied once a set of reasonable generators is determined for $A_n(V)$ to aid in determining the relations of those generators, such as using the relations of those generators in the lower level Zhu algebras and the zero mode actions on $V$-modules induced from those lower level Zhu algebras.  We prove that the condition that $(L(-1) + L(0))v$ acts as zero in $A_n(V)$ for $n \in \mathbb{Z}_+$ and for all $v$ in $V$ is a necessary added condition in the definition of the Zhu algebra at level higher than zero.  We discuss how these results on generators and relations apply to the level $n$ Zhu algebras for the Heisenberg vertex operator algebra and the Virasoro vertex operator algebras at any level $n \in \mathbb{N}$.
\end{abstract}

\maketitle

\section{Introduction}
Given a vertex operator algebra $V$, Zhu introduced in \cite{Z}, an associative algebra  which was used to prove modular invariance of graded characters for $V$-modules in the case when $V$ is $C_2$-cofinite and rational.   In \cite{DLM},  Dong, Li, and Mason generalized these algebras, defining for every vertex operator algebra $V$, a family of associative algebras denoted $A_n(V)$, for $n\in \mathbb{N}$.  The $n=0$ algebra in this family is the original (level zero) associative algebra that Zhu introduced, and the algebras for $n>0$ are now known as ``higher level Zhu algebras".  These Zhu algebras are powerful tools in the study of both the representation theory of vertex operator algebras and the number-theoretic properties of the graded traces and ``pseudo-traces" of these modules.  However a major obstacle in realizing the full potential of these Zhu algebras is that they are extremely difficult to calculate. It is thus an important problem to develop methods for and insight into the construction of these algebras.

In this paper, we  prove general results regarding the generators and relations for the level $n$ Zhu algebra, $n \in \mathbb{N}$, for a vertex operator algebra $V$, and show how these results apply to the level $n$ Zhu algebras for the Heisenberg and Virasoro vertex operator algebras, as well as to the conformal vector or Heisenberg-like vectors in any vertex operator algebra.   We also give some clarifications on the necessity of a certain condition in the definition of the higher level Zhu algebras, i.e. for $n >0$. 

As was mentioned above, Zhu algebras are important tools in studying the representation theory of a vertex operator algebra.  In  \cite{FZ}, Frenkel and Zhu proved that there is a one-to-one correspondence between isomorphism classes of irreducible representations of a vertex operator algebra $V$ and isomorphism classes of irreducible representations of its (level zero) Zhu algebra $A_0(V)$.   However if $V$ admits indecomposable non irreducible modules (i.e., if $V$ is irrational), then $A_0(V)$ will not in general be sufficient for the correspondence between isomorphism classes of indecomposable modules for $V$ and for $A_0(V)$ to hold.   In \cite{DLM}, Dong, Li, and Mason define a functor, denoted $\Omega_n$, from the category of admissible $V$-modules, to the category of $A_n(V)$ modules and another functor, denoted $L_n$ from the category of $A_n(V)$-modules to the category of admissible $V$-modules and (under an added necessary condition provided in \cite{BVY}) prove that these functors realize a bijection between the isomorphism classes of irreducibles  for $V$ and for $A_n(V)$.   In \cite{BVY}, the second author along with Vander Werf and Yang  prove several results in the theory of higher level Zhu algebras which help clarify some statements of \cite{DLM} and further illuminate the relationship between  $V$-modules and $A_n(V)$-modules in the non semi-simple setting of indecomposable non irreducible modules for irrational vertex operator algebras. In particular, in \cite{BVY} the nature of the level $n$ Zhu algebra in relation to the level $n-1$ Zhu algebra was shown to determine whether new information about certain types of modules can be gleaned from the level $n$ Zhu algebra not already detected by the level $n-1$ Zhu algebra, and the nature of that information.

The second major application of higher level Zhu algebras, is to construct the graded pseudo-traces and analyze their number-theoretic properties.   An analogue of Zhu's modular invariance for graded characters (i.e., graded traces) for modules for a rational $C_2$-cofinite vertex operator algebra was provided by Miyamoto in \cite{Miyamoto2004}, where the notion of a graded ``pseudo-trace" was defined for irrational $C_2$-cofinite vertex operator algebras. While the graded characters of modules for rational vertex operator algebras encode information about the $L(0)$ eigenspaces of the modules, these graded pseudo-traces give information about the generalized eigenspaces. Where the original (level $0$) Zhu algebra plays an important role in Zhu's proof of modular invariance in the rational setting, the higher level Zhu algebras are central to Miyamoto's construction. 

In \cite{AN}, Arike and Nagatomo studied graded pseudo-traces for the $\mathbb{Z}_2$-orbifold of  symplectic fermions. While their construction notably does not use Zhu algebras, it is not known if their graded pseudo-traces coincide with those of Miyamoto. To the authors' knowledge, no examples of graded pseudo-traces have been explicitly determined using Miyamoto's construction, where the higher level Zhu algebras are necessary. Examples of these pseudo-traces using Miyamoto's construction will be provided by the second author and their collaborators in subsequent publications.

A major obstacle in utilizing higher level Zhu algebras in the study of vertex operator algebras and related questions of modular invariance, is that determining their structure is incredibly difficult. In \cite{BVY-Heisenberg} and \cite{BVY-Virasoro}, the second author along with Vander Werf and Yang gave the first explicit constructions of higher level Zhu algebras, determining the structure of $A_1(V)$ for $V$ the rank one Heisenberg vertex operator algebra and the Virasoro vertex operator algebra, respectively. They also proved that the level one Zhu algebra for the Virasoro vertex operator algebra can be used to induce new modules not induced by the level zero Zhu algebra. Later, in \cite{AB-details-of-Heisenberg}, the present authors determined the structure of the level two Zhu algebra for the Heisenberg vertex operator algebra; see Theorem \ref{A_2-theorem} below.   Higher level twisted and untwisted Zhu algebras were studied in \cite{V} by van Ekeren in the setting of affine and Virasoro vertex operator algebras.  And in \cite{C}, \^Ceperi\'c calculated the level one Zhu algebra for the fixed point subalgebra of symplectic fermions.

Notably in the constructions provided in \cite{BVY-Heisenberg}, \cite{BVY-Virasoro}, and \cite{AB-details-of-Heisenberg}, only very minimal facts about modules for the relevant vertex operator algebras are used. The purpose of this paper is to expand on the insight gained from these papers to provide tools for determining the structure of higher level Zhu algebras. In fact, the results of this paper require no knowledge of the modules of any of the vertex operator algebras to which our results apply. This is particularly important, as one would like to use higher level Zhu algebras to study the modules of vertex operator algebras for which very little information is known, e.g. irrational $C_2$-cofinite vertex operator algebras which are central to logarithmic conformal field theory but for which very few examples are currently known.

In this paper, we present results on generators and relations for higher level Zhu algebras that depend only on the definition of higher level Zhu algebras and, for some results, some relatively mild assumptions about the vertex operator algebra at hand. The main results of this paper are presented in Sections \ref{general-n-section} and \ref{definition-clarification-section}.  In the first subsection of Section \ref{general-n-section},  we prove a recursion formula for elements in $A_n(V)$ which applies to any vertex operator algebra $V$, and provides relations which are key to calculating $A_n(V)$ explicitly, in particular in determining minimal generating sets.  In the  second subsection of Section \ref{general-n-section}, we prove further results on generators and relations for $A_n(V)$, which apply to  a vertex operator algebra $V$ which has generators with certain nice properties. We give a  general formula for multiplication in higher level Zhu algebras for certain elements of the vertex operator algebra.  In particular, we prove that terms in multiplication formulas in the level $n$ Zhu algebras for these vertex operator algebra elements are indexed by partitions of $2n$.

In Section \ref{definition-clarification-section}, we provide an important clarification about the construction of higher level Zhu algebras and discuss the implications of this clarification for future constructions of these algebras. We note that in the definition of the level $n$ Zhu algebra $A_n(V)= V/O_n(V)$ for $n \in\mathbb{Z}_+$  in some instances, cf. \cite{Miyamoto2004}, the ideal $O_n(V)$ is defined to be span$\{ u \circ_n v \; | \; u,v \in V\}$ and it is implicitly assumed that this ideal contains the elements of the form $(L(-1) + L(0))v$ for $v\in V$.  This assumption is likely due to the fact that in the original Zhu algebra, i.e., the $n=0$ Zhu algebra, the vectors $(L(-1)+L(0))v\in \mathrm{span} \{ u \circ_n v \; | \; u,v \in V\}$ for all $v\in V$, see Remark \ref{first-remark}. In this paper, we prove that for Zhu algebras of level $n>0$, it is a necessary extra condition to include the vectors $(L(-1) + L(0))v$ for $v\in V$ in the definition of $O_n(V)$, as they are not in general in span$\{ u \circ_n v \; | \; u,v \in V\}$.  In addition, we give certain general cases for when these vectors will not be in this spanning set and thus give new relations in $A_n(V)$.  

In Section 5, we highlight some aspects of Zhu algebras that provide techniques for calculating higher level Zhu algebras.  These techniques include utilizing the structure of  the level $n-1$ Zhu algebra and the $V$-modules induced at level $n-1$ to help determine the relations for the level $n$  Zhu algebra.

In Section 6, we give applications of our results for the vertex operator algebras associated to the Heisenberg and Virasoro algebras, respectively, recall the cases for which the higher level Zhu algebras have been explicitly calculated for these vertex operator algebras, and present a conjecture from \cite{AB-details-of-Heisenberg} about the structure for all higher level Zhu algebras for the Heisenberg vertex operator algebra.

{\bf Acknowledgments: } The authors thank the Simons Foundation for their support. The first author was a recipient of an AMS-Simons Foundation Travel Grant, and the second author was the recipient of a Simons Foundation Collaboration Grant 282095.  

{\bf Notation:} Throughout this paper, $\mathbb{N}$ denotes the nonnegative integers and $\mathbb{Z}_+$ denotes the positive integers.

\section{The algebras $A_n(V)$, and the functors $\Omega_n$ and $L_n$}\label{zhu-algebra-definition-section}

In this section, letting $V$ be a vertex operator algebra, we recall the definition and some properties of the algebras $A_n(V)$ for $n \in \N$, first introduced in \cite{Z} for $n = 0$, and then generalized to $n >0$ in \cite{DLM}.   We then recall the functors $\Omega_n$ and $L_n$ defined in \cite{DLM} and some results from \cite{BVY}.

For $n \in \N$, let $O_n(V)$ be the subspace of $V$ spanned by elements of the form
\begin{equation}\label{define-circ}
u \circ_n v =
\res_x \frac{(1 + x)^{\mathrm{wt}\, u + n}Y(u, x)v}{x^{2n+2}}
\end{equation}
for all homogeneous $u \in V$ and for all $v \in V$, and by elements of the form $(L(-1) + L(0))v$ for all $v \in V$. The vector space $A_n(V)$ is defined to be the quotient space $V/O_n(V)$.

That is, using the notation introduced in \cite{BVY},  we define 
\begin{equation}\label{define-O-subsets} 
O^L (V) = \{ (L(-1) + L(0))v \, | \, v \in V \} \quad
\mathrm{and}
\quad O_n^\circ(V) = \mathrm{span} \{ u \circ_n v \; | \; u,v \in V\} ,
\end{equation} 
then
\begin{equation}\label{define-O}  
O_n(V) =  O^L (V) + O_n^\circ(V) .
\end{equation}
We will sometimes refer to $\circ_n$ as the ``circle product".

\begin{rem}\label{first-remark}{\em
For $n=0$, since $v \circ_0 \mathbf{1} = v_{-2} \mathbf{1} + (\mathrm{wt} \, v) v = L(-1) v + L(0) v$, it 
follows that $(L(-1) + L(0)) v \in O^\circ_0(V)$ for all $v \in V$ and thus $O_0(V) = O_0^\circ(V)$. 
However, we show in Section \ref{definition-clarification-section} that for $n > 0$, in 
general, $O^L (V) \not\subset O_n^\circ(V)$.  Thus for $A_n(V)$ to have a well-defined action on 
modules via zero modes,  it is necessary to not just define $O_n(V)$ to consist of $O_n^\circ(V)$ as is
 sometimes done, but to also include $O^L(V)$ as part of the definition; cf. \cite{DLM}, 
 \cite{Miyamoto2004}.  This is simply due to the fact that from the $L(-1)$-derivative property, the zero
 mode of $L(0)v$ which is $(\mathrm{wt} \, v)v_{\mathrm{wt} \, v - 1}$ is equal to the negative of the zero mode of $L(-1)v$.  We show in Section 4 that in general $O^L (V) \not\subset O_n^\circ(V)$, by showing that, for instance, any time $V$ is strongly generated (as defined in Section \ref{generators-section}) by an element $u$ such that $1 \leq \mathrm{wt} \, u \leq n$, then $(L(-1) + L(0) )u \notin O_n^\circ(V)$.  More generally, for any $V$ with  $u \in V$ of positive weight or of weight zero such that $u_{-2} {\bf 1} \neq 0$, and such that the vertex subalgebra generated by $u$ is strongly generated, then one can have $(L(-1) + L(0) )u \notin O_n^\circ(V)$ if $n >0$.    For instance, this holds for the Heisenberg vertex operator algebra generated by a weight one element $u = \alpha(-1) \mathbf{1}$, or for such weight one elements $u \in V$ where there is not linear dependence between $(L(-1) + L(0))u$ and circle products of vectors in $V$ not generated by $u$.  This also holds for the vertex operator subalgebra of $V$ generated by the conformal element if $n>1$.   And finally we note that our result also pertains to the case when $V$ has a nontrivial vacuum (weight zero) space containing any vectors $u$ such that $u_{-2} {\bf 1} \neq 0$.  Thus there are common settings in which a vertex operator algebra $V$ has the property that $O^L (V) \not\subset O_n^\circ(V)$.  More  importantly, as we point out in Remark \ref{OL-remark}  below, the presence of these $(L(-1) + L(0)) v \in 
O_n(V)$ is used heavily in the constructions of the higher level Zhu algebras in \cite{BVY-Heisenberg}, \cite{BVY-Virasoro},  and \cite{AB-details-of-Heisenberg}. }
\end{rem}

We define the following multiplication on $V$
\[
u *_n v = \sum_{m=0}^n(-1)^m\binom{m+n}{n}\res_x \frac{(1 + x)^{\mathrm{wt}\, u + n}Y(u, x)v}{x^{n+m+1}},
\]
for $v \in V$ and homogeneous $u \in V$, and for general $u \in V$, $*_n$ is
defined by linearity.   It is shown in \cite{DLM} that with this multiplication, the subspace $O_n(V)$ of $V$ is a two-sided ideal of $V$, and $A_n(V)$ is an associative algebra, called the {\it level $n$ Zhu algebra}.

\begin{rem} {\em Due to the truncation condition on $V$, the expressions $u \circ_n v$ and $u *_n v$ in the definition of $A_n(V)$ are indeed well-defined elements of $V$, even if $\mathrm{wt} \, u + n <0$, i.e. they are finite sums of elements in $V$.  Thus, it is unnecessary (as is often done in the literature) to assume that $V$ is positively graded.  In fact, notably many of our results in this paper apply to elements in $V$ of negative weight.}
\end{rem}

Next we prove a lemma stating that certain elements in $V$, namely those of the form (\ref{more-elements}) given below, lie in $O_n(V)$, and in fact, more precisely in $O_n^\circ(V)$.  In \cite{Z}, Zhu proved that such elements lie in $O_n(V)$ in the case when $n = 0$, and his proof extends easily to the $n>0$ case.  In \cite{DLM} this result without proof was used to show that there is a surjective homomorphism from $A_n(V)$ onto $A_{n-1}(V)$ for any $n \in \mathbb{Z}_+$.  Here we provide the proof so as to point out a refinement of the result given in \cite{DLM}, namely that these elements lie not just in $O_n(V)$, but in $O_n^\circ(V)$.  We will use this more precise statement later in Section \ref{general-n-section}, in particular to prove the useful recursion formula in Proposition \ref{recursion-n}.

\begin{lem}\label{Zhu-lemma}
For every homogeneous element $u \in V$ and for $m, k \in \mathbb{N}$ with $m \geq k \geq 0$, elements of the form
\begin{equation}\label{more-elements}
\res_x \frac{(1 + x)^{\mathrm{wt}\, u + n+ k}Y(u, x)v}{x^{m+2n+2}}
\end{equation}
lie not only in $O_n(V)$, but in $O_n^\circ(V)$.  
\end{lem}

 \begin{proof}
 The proof follows from the definition of $O_n^\circ(V)$ and the $L(-1)$-derivative property for $V$.  First we note that $m\geq k$ implies that $m-j \geq 0$ for $j =0,\dots,k$, and thus since 
\begin{equation}
\res_x \frac{(1 + x)^{\mathrm{wt}\, u + n+ k}Y(u, x)v}{x^{m+2n+2}} = \sum_{j=0}^k
\res_x \binom{k}{j} \frac{(1 + x)^{\mathrm{wt}\, u + n}Y(u, x)v}{x^{m+2n+2-j}}
\end{equation}
we need only prove the statement for the case when $k = 0$ and $m\geq 0$.  

We then proceed by induction on $m$.  For $k= 0$ and $m = 0$, the statement follows immediately from the definition of $\circ_n$.  Assume the statement holds for $k = 0$ and $m= j$.   Then for $m = j+1$, by the inductive assumption, we have that 
\begin{equation}\label{inductive-assumption} 
 \res_x \frac{(1 + x)^{  \mathrm{wt}\, w + n}Y(w, x)v}{x^{j+2n+2}}  \in O_n^\circ(V) \quad \mbox{ for both $w = u$ and $w = L(-1)u$,}
 \end{equation}
where we note that $\mathrm{wt} \, L(-1) u = \mathrm{wt} \, u + 1$.  Then by  the definition of $\res_x$, and the $L(-1)$-derivative property in $V$ , we have
\begin{eqnarray*}
0 &=& \res_x \frac{d}{dx} \left( \frac{(1 + x)^{\mathrm{wt}\, u + n  + 1}Y(u, x)v}{x^{j+2n+2}}\right) \\
&=&  \res_x (\mathrm{wt} \, u + n  + 1)   \frac{(1+x)^{ \mathrm{wt} \, u + n  } Y(u, x)v}{x^{j+2n+2}} - \res_x 
(j + 2n + 2) \frac{(1 + x)^{\mathrm{wt}\, u + n  + 1} Y(u, x)v}{x^{j+2n+3}}  \\
& & \quad + \, \res_x \frac{(1 + x)^{\mathrm{wt}\, u + n  + 1}\frac{d}{dx}Y(u, x)v}{x^{j+2n+2}} \\
&=& \res_x (\mathrm{wt} \, u + n + 1 )   \frac{(1+x)^{ \mathrm{wt} \, u + n } Y(u, x)v}{x^{j+2n+2}} -  (j + 2n + 2)  \left( \res_x 
\frac{(1 + x)^{\mathrm{wt}\, u + n} Y(u, x)v}{x^{(j+1)+2n+2}}  \right.  \\
& & \quad \left.  + \,  \res_x 
\frac{(1 + x)^{\mathrm{wt}\, u + n} Y(u, x)v}{x^{j+2n+2}} \right) 
+  \res_x \frac{(1 + x)^{\mathrm{wt}\, u + n  + 1 }Y(L(-1) u, x)v}{x^{j+2n+2}} \\
 &=&  (\mathrm{wt} \, u -j-n-1 )  \res_x   \frac{(1+x)^{ \mathrm{wt} \, u + n } Y(u, x)v}{x^{j+2n+2}} -  (j + 2n + 2)   \res_x 
\frac{(1 + x)^{\mathrm{wt}\, u + n} Y(u, x)v}{x^{(j+1)+2n+2}}   \\
& &   \quad + \, \res_x \frac{(1 + x)^{\mathrm{wt}\, u + n  + 1}Y(L(-1) u, x)v}{x^{j+2n+2}} . 
\end{eqnarray*}
By Eqn. \eqref{inductive-assumption}, the first and last terms of the right hand side are in $O^\circ_n(V)$ , and thus
\[  (j + 2n + 2)  \res_x  \frac{(1 + x)^{\mathrm{wt}\, u + n} Y(u, x)v}{x^{(j+1)+2n+2}} \in O_n^\circ(V).\]
Since $j + 2n + 2 >0$, the result holds.  
 \end{proof}

Lemma \ref{Zhu-lemma} implies that $O^\circ_n(V) \subset O^\circ_{n-1}(V)$, and thus $O_n(V) \subset O_{n-1}(V)$.  In fact, from Proposition 2.4 in \cite{DLM}, we have that the map 
\begin{eqnarray}\label{surjection}
A_n(V) & \longrightarrow & A_{n-1}(V) \\
v + O_n(V) & \mapsto & v + O_{n-1}(V) \nonumber
\end{eqnarray}
is a surjective algebra homomorphism.

Next we recall the definitions of various $V$-module structures and then briefly introduce the functors that associate modules for $A_n(V)$ with certain types of modules for $V$.  We assume the reader is familiar with the notion of weak $V$-module for a vertex operator algebra $V$ (cf. \cite{LL}).  

\begin{defn}\label{N-gradable-definition}
{\em An {\it $\mathbb{N}$-gradable weak $V$-module} (also often called an {\it admissible $V$-module} as in \cite{DLM}) $W$ for a vertex operator algebra $V$ is a weak $V$-module that is $\mathbb{N}$-gradable, $W = \coprod_{k \in \mathbb{N}} W(k)$, with $v_m W(k) \subset W(k + \mathrm{wt} v - m -1)$ for homogeneous $v \in V$, $m \in \mathbb{Z}$ and $k \in \N$, and without loss of generality, we can and do assume $W(0) \neq 0$, unless otherwise specified.  We say elements of $W(k)$ have {\it degree} $k \in \mathbb{N}$.

An {\it $\mathbb{N}$-gradable generalized weak $V$-module} $W$ is an $\mathbb{N}$-gradable weak $V$-module that admits a decomposition into generalized eigenspaces via the spectrum of $L(0) = \omega_1$ as follows: $W=\coprod_{\lambda \in{\C}}W_\lambda$ where $W_{\lambda}=\{w\in W \, | \, (L(0) - \lambda \, id_W)^j w= 0 \ \mbox{for some $j \in \mathbb{Z}_+$}\}$, and in addition, $W_{n +\lambda}=0$ for fixed $\lambda$ and for all sufficiently small integers $n$. We say elements of $W_\lambda$ have {\it weight} $\lambda \in \mathbb{C}$.

A {\it generalized $V$-module} $W$ is an $\mathbb{N}$-gradable generalized weak $V$-module where $\dim W_{\lambda}$ is finite for each $\lambda \in \mathbb{C}$.   

An {\it (ordinary) $V$-module} is a  generalized $V$-module such that  the generalized eigenspaces $W_{\lambda}$ are in fact eigenspaces, i.e., $W_{\lambda}=\{w\in W \, | \, L(0) w=\lambda w\}$.}
\end{defn}

We will often omit the term ``weak" when referring to $\mathbb{N}$-gradable weak and $\mathbb{N}$-gradable generalized weak $V$-modules.   

The term {\it logarithmic} is also often used in the literature to refer to $\mathbb{N}$-gradable weak generalized modules  or generalized modules. 

Let $W$ be an $\mathbb{N}$-gradable $V$-module, and let
\begin{equation}\label{Omega}
\Omega_n(W) = \{w \in W \; | \; v_iw = 0\;\mbox{if}\; \wt v_i < -n \; 
\mbox{for $v\in V$ of homogeneous weight}\}.
\end{equation}
It was shown in \cite{DLM} that $\Omega_n(W)$ is an $A_n(V)$-module
via the action $o(v+O_n(V)) = v_{\mathrm{wt} \, v -1}$ for $v \in V$.   In particular, this action satisfies $o(u *_n v) = o(u)o(v)$ for $u,v \in A_n(V)$.

Furthermore, it was shown in \cite{DLM} and \cite{BVY} that there is a bijection between the isomorphism classes of irreducible $A_n(V)$-modules which cannot factor through $A_{n-1}(V)$ and the isomorphism classes of irreducible $\mathbb{N}$-gradable $V$-modules with nonzero degree $n$ component.  

The functor $L_n$ from the category of $A_n(V)$-modules to the category of $\mathbb{N}$-gradable $V$-modules is more complicated to define and thus we refer the reader to \cite{DLM} and \cite{BVY} for the definition.

We have the following theorem from \cite{BVY}:

\begin{thm}\label{mainthm}\cite{BVY}
For $n \in \N$, let $U$ be a nonzero $A_n(V)$-module such that if $n>0$, then $U$ does not factor through $A_{n-1}(V)$. Then $L_n(U)$ is an $\mathbb{N}$-gradable $V$-module with $L_n(U)(0) \neq 0$.  If we assume further that there is no nonzero submodule of $U$ that factors through $A_{n-1}(V)$, then $\Omega_n/\Omega_{n-1}(L_n(U)) \cong U$.
\end{thm}

One of the main reasons we are interested in Theorem \ref{mainthm} is what it implies about when modules for the higher level Zhu algebras give rise to indecomposable nonsimple modules for $V$ not detected by the lower level Zhu algebras, {cf. \cite{BVY}, \cite{BVY-Heisenberg}, \cite{BVY-Virasoro}. A major motivation for the present paper is to provide methods for construction of those higher level Zhu algebras which can be used to induce new indecomposable nonsimple modules.

\section{Some general results for generators and relations of the algebra $A_n(V)$}\label{general-n-section}

In this section we prove the main results of the paper regarding generators and relations 
 that occur in $A_n(V)$ in certain settings.  In Section \ref{recursion-section}, we prove a general relation that holds for all elements in a vertex operator algebra, as well as a recursion formula that holds for ``most" elements in a vertex operator algebra.  In Section \ref{generators-section} we use this recursion formula to identify sets of generators for $A_n(V)$ if $V$ is a strongly generated cyclic vertex operator algebra, such as the Heisenberg or Virasoro vertex operator algebras, but this also applies for broad classes of elements in a more general vertex operator algebra.  

Recall the notation (\ref{define-O-subsets}) and (\ref{define-O}).  To distinguish between relations in $A_n(V)$ arising from $O^L(V)$ versus $O^\circ_n(V)$  in our results below, we use the notation $u \approx v$ if $u \equiv v \, \mathrm{mod} \,O^L(V)$.  And we will write $u \sim_n v $ if  $u \equiv v  \, \mathrm{mod} \, O_n^\circ(V)$.  More broadly if $u$ is equivalent to $w$ modulo $O_n(V)$, we will write $u \equiv_n w$.  

 \subsection{General reduction relations in $A_n(V)$ arising from $O^L(V)$ and $O_n^\circ(V)$}\label{recursion-section}
 
 We have the following reduction relation in $A_n(V)$ arising from $O^L(V)$.
 
 \begin{lem}\label{L-reduction-lemma}
 Let  $V$ be a vertex operator algebra, $r \in \mathbb{N}$, $u^{(1)}, \dots, u^{(r)} \in V$, and $ k_1, \dots, k_r \in \mathbb{Z}$.  For $w = u_{k_1}^{(1)} u_{k_2}^{(2)} \cdots u_{k_r}^{(r)} \mathbf{1} \in V$, we have 
\begin{equation}
 (\mathrm{wt} \,  w )u_{k_1}^{(1)} u_{k_2}^{(2)} \cdots u_{k_r}^{(r)} \mathbf{1} \approx
 \sum_{j = 1}^r  k_j  u_{k_1}^{(1)} u_{k_2}^{(2)} \cdots u_{k_{j-1}}^{(j-1)} u_{k_j - 1}^{(j)} u_{k_{j+1}}^{(j+1)} \cdots u_{k_r}^{(r)} \mathbf{1} .
\end{equation}
In particular, for $u,v \in V$ and $j,k \in \mathbb{Z}$, $j \neq 0$, we have
\begin{equation}\label{L-reduction-special-case} 
u_{-j-1} v_{-k} {\bf 1} \approx - \left( \frac{\mathrm{wt} \, u  + \mathrm{wt} \, v  +j  +k  -2}{j} \right) u_{-j} v_{-k} {\bf 1}  - \frac{k}{j} u_{-j} v_{-k-1} {\bf 1} .
\end{equation}
 \end{lem}

\begin{proof} From the $L(-1)$-derivative property and the $L(-1)$-bracket derivative property for $V$, we have that for $k \in \mathbb{Z}$ and $v \in V$.
\begin{equation}\label{L(-1)-properties}
\left[ L(-1), v_{k} \right] = (L(-1) v)_{k} = -k v_{k-1}.
\end{equation}
Thus we have for $w = u_{k_1}^{(1)} u_{k_2}^{(2)} \cdots u_{k_r}^{(r)} \mathbf{1}$ 
\begin{eqnarray*}
0 &\approx& (L(-1) + L(0))  u_{k_1}^{(1)} u_{k_2}^{(2)} \cdots u_{k_r}^{(r)} \mathbf{1}\\
&=& (\mathrm{wt} \,  w )u_{k_1}^{(1)} u_{k_2}^{(2)} \cdots u_{k_r}^{(r)} \mathbf{1}
- k_1  u_{k_1 -1}^{(1)} u_{k_2}^{(2)} \cdots u_{k_r}^{(r)} \mathbf{1} \\
& & \quad - \,  k_2 u_{k_1}^{(1)} u_{k_2-1}^{(2)} \cdots u_{k_r}^{(r)} \mathbf{1} - \cdots  -k_r u_{k_1}^{(1)} u_{k_2}^{(2)} \cdots u_{k_r-1}^{(r)} \mathbf{1} \\
&=& (\mathrm{wt} \,  w )u_{k_1}^{(1)} u_{k_2}^{(2)} \cdots u_{k_r}^{(r)} \mathbf{1}
- \sum_{j = 1}^r  k_j  u_{k_1}^{(1)} u_{k_2}^{(2)} \cdots u_{k_{j-1}}^{(j-1)} u_{k_j - 1}^{(j)} u_{k_{j+1}}^{(j+1)} \cdots u_{k_r}^{(r)} \mathbf{1} 
\end{eqnarray*}
proving the first equation.

The second equation follows from the first by letting $r = 2$, $u^{(1)} = u$, $u^{(2)} = v$, $k_1 = -j$, and $k_2 = -k$.  
\end{proof}

For certain elements of $V$, we have the following  relations in $A_n(V)$ arising from a recursion in $O_n^\circ(V)$.

\begin{prop}\label{recursion-n}
Let $V$ be a vertex operator algebra and fix $n \in \mathbb{N}$.  

If  $u \in V$ satisfies $\mathrm{wt} \, u > -n$, then for $v \in V$ and $m \geq n+ 1$
\begin{equation}\label{recursion-n-equation}
u_{\mathrm{wt} \, u - m - 1}v \ \sim_n  \ (-1)^{m+ \mathrm{wt} \, u} \sum_{j = 1}^{n+ \mathrm{wt} \, u} \binom{m-n-1}{j-1} \binom{m-n-j-1}{n+\mathrm{wt} \, u -j } u_{\mathrm{wt} \, u - n - j - 1} v
\end{equation}
where $\sim_n$ denotes equivalence modulo $O^\circ_n(V)$.  In particular, this gives $u_{\mathrm{wt} \, u - m - 1} v \ \sim_n \ u_{\mathrm{wt} \, u - m - 1} v$, and thus no new relations in $A_n(V)$ for $n+1 \leq m \leq 2n + \mathrm{wt} \, u$, but does give additional relations in $A_n(V)$ arising from $O_n^\circ(V)$ for $m > 2n + \mathrm{wt} \, u$.

If  $u \in V$ satisfies $\mathrm{wt} \, u = -n$,  then for $v \in V$, and $m \geq n+ 1$, 
\begin{equation}\label{recursion-n-equation-lowest}
u_{\mathrm{wt} \, u - m -1}v  \sim_n 0.
\end{equation}
\end{prop}

\begin{proof}
For $\mathrm{wt} \, u > -n$, and  $n+1 \leq m \leq 2n + \mathrm{wt} \, u$, if $j \neq m - n$ on the righthand side of Eqn.\ (\ref{recursion-n-equation}) then one of the two binomial coefficients for that term in the sum is zero.  Whereas if $j = m-n$, the coefficient of that term is one.  Thus for $n+1 \leq m \leq 2n + \mathrm{wt} \, u$, the righthand side of this equation reduces to $u_{\mathrm{wt}\, u - m - 1}v$, giving a trivial relation in $A_n (V)$ in this case, as stated.

By (\ref{more-elements}) with $k = 0$, and letting $m = l \geq 0$ in Eqn.\ (\ref{more-elements}), we have 
\begin{eqnarray*}
0 &\sim_n& \mathrm{Res}_x x^{-l - 2n-2} (1+x)^{\mathrm{wt} \, u + n} \sum_{ j \in \mathbb{Z}} u_j x^{-j-1} v \\
&=& \mathrm{Res}_x \sum_{ j \in \mathbb{Z}} \sum_{ i \in \mathbb{N}} \binom{\mathrm{wt} \, u + n}{i} x^{i-l-2n-j-3} u_j v \\
&=& \sum_{i \in \mathbb{N}}  \binom{\mathrm{wt} \, u + n}{i} u_{i- l - 2n - 2}v.
\end{eqnarray*}
Thus for $\mathrm{wt} \, u + n > 0$ and $l \geq 0$, we have 
\begin{equation}\label{for-recursion-proof}
u_{-l-2n-2}v \ \sim_n \ -  \sum_{i = 1}^{n + \mathrm{wt} \, u}  \binom{n + \mathrm{wt} \, u}{i} u_{i - l - 2n - 2}v,
\end{equation}
and for $\mathrm{wt} \, u = -n$ and $l\geq 0$, we have $u_{-l-2n-2} v \sim_n 0$.  This finishes the proof of the last statement of the proposition, i.e. the case when $\mathrm{wt} \, u = -n$.

We are left with proving Eqn.\ (\ref{recursion-n-equation}) holds when $\mathrm{wt} \, u > -n$ and $m > 2n + \mathrm{wt} u$.   Since the equation holds for $m = 2n + \mathrm{wt} \, u$, we make the inductive assumption that Eqn.\ (\ref{recursion-n-equation}) holds for $u_{\mathrm{wt} \, u -m' - 1}$ for $2n+ \mathrm{wt} \, u \leq m' < m$.  Then by Eqn.\ (\ref{for-recursion-proof}) with $l = m - \mathrm{wt} \, u - 2n - 1$ (which is greater than or equal to zero when $m>2n+\mathrm{wt}\, u$), and the inductive assumption, as well as the fact that we already have shown that Eqn.\ (\ref{recursion-n-equation}) holds for $u_{\mathrm{wt} \, u -m' - 1}$ when $n+ 1 \leq m' \leq 2n+ \mathrm{wt} \, u$, we have 
\begin{eqnarray*}
\lefteqn{u_{\mathrm{wt} \, u - m - 1}v \ = u_{-l - 2n  - 2}v }\\
& \sim_n & -  \sum_{i= 1}^{n+ \mathrm{wt} \, u}  \binom{n + \mathrm{wt} \, u}{i} u_{i- l - 2n - 2}v\\
&=& -  \sum_{i = 1}^{n + \mathrm{wt} \, u}  \binom{n + \mathrm{wt} \, u}{i} u_{\mathrm{wt} \, u - (m-i) - 1}v \\
&\sim_n& -  \sum_{i = 1}^{n + \mathrm{wt} \, u}  \binom{n + \mathrm{wt} \, u}{i} 
(-1)^{m-i+\mathrm{wt} \, u} \sum_{j= 1}^{n+ \mathrm{wt} \, u} \binom{m-i-n-1}{j-1} \binom{m-i-n-j-1}{n+\mathrm{wt} \, u - j} \\
& & \qquad u_{\mathrm{wt} \, u - n - j - 1} v\\
&=&  (-1)^{m+\mathrm{wt} \, u} \sum_{j= 1}^{n+ \mathrm{wt} \, u}  \ \sum_{i = 1}^{ n + \mathrm{wt} \, u }  (-1)^{i+1} \binom{n + \mathrm{wt} \, u }{i} 
\binom{m-i-n-1}{j-1} \binom{m-i-n-j-1}{n+\mathrm{wt} \, u - j} \\
& & \qquad u_{\mathrm{wt} \, u - n - j - 1} v.
\end{eqnarray*}
Thus the result holds if we can show that for $j = 1, 2, \dots, n + \mathrm{wt} \, u$, and $m > 2 n + \mathrm{wt} \, u$
\begin{multline*}
 \sum_{i = 1}^{ n + \mathrm{wt} \, u }  (-1)^{i+1} \binom{n + \mathrm{wt} \, u}{i} 
\binom{m-i-n-1}{j-1} \binom{m-i-n-j-1}{n+\mathrm{wt} \, u - j }  \\
= \binom{m-n-1}{j-1} \binom{m-n-j-1}{n+\mathrm{wt} \, u - j } 
\end{multline*}
or equivalently that for $j = 1, \dots, n + k$, with $k  = \mathrm{wt} \, u$, and  $m > 2n + k$ that 
\begin{equation}\label{combinatorial-formula}
 \sum_{i = 0}^{ n + k}  (-1)^{i+1} \binom{ n + k }{i} 
\binom{m-i-n-1}{j-1} \binom{m-i-n-j-1}{n+ k-j}  =0.
\end{equation}

Since the proof of this last equality is quite long, and only a combinatorial identity, we leave it for the appendix; see Proposition \ref{recursion-n-identity}.
\end{proof}

The recursion (\ref{recursion-n-equation}) can be used to determine certain elements in $O_n(V)$ that are useful in determining the generators and relations for $A_n(V)$.  For instance, 
letting $Y^+(v,x) \in (\mathrm{End}(V))[[x]]$ and $Y^-(v,x)\in x^{-1}(\mathrm{End}(V))[[x^{-1}]]$ denote the regular and singular parts, respectively, of $Y(v,x)$ for $v \in V$, one can use the recursion to investigate which elements are of the form $u \circ_n v$ for $u,v \in V$, and thus lie in $O_n(V)$.  To that end, we note the following lemma.

\begin{lem}\label{Y+-lemma} Let $V$ be a vertex operator algebra and fix $n \in \mathbb{N}$.  If $u \in V$ satisfies wt $u > -n$, then for $v \in V$, we have 
\begin{multline}\label{Y+}
Y^+(u,x)v \sim_n 
 \sum_{k=1}^{n - \mathrm{wt} \, u + 1}  u_{-k} v x^{k-1} \\
 +   \sum_{m\ge n+1} \sum_{k = n - \mathrm{wt} \, u + 2}^{2n + 1} (-1)^{m- \mathrm{wt} \, u}  \binom{m-n-1}{k-n + \mathrm{wt} \, u -2} \binom{m-k-\mathrm{wt} \, u}{2n-k + 1 } u_{-k} v x^{-\mathrm{wt} \, u  + m}  ,
\end{multline}
and if $\mathrm{wt} \,  u = -n$, then
\begin{equation}\label{Y+-neg}
Y^+(u,x)v \sim_n 
 \sum_{k=1}^{2n + 1}  u_{-k} v x^{k-1} .
 \end{equation}
\end{lem}

\begin{proof}
Applying the recursion given in Eqn.\ (\ref{recursion-n-equation}), we have that if $\mathrm{wt} \, u > -n$, then
\begin{eqnarray*}
\lefteqn{Y^+(u, x) v }\\
&=& \sum_{k\ge1}u_{-k}vx^{k-1} =  \sum_{m\ge \mathrm{wt} \, u}u_{\mathrm{wt} \, u - m - 1}vx^{-\mathrm{wt} \, u  + m} \\
&=&   \sum_{m=\mathrm{wt} \, u}^n u_{\mathrm{wt} \, u - m - 1}vx^{-\mathrm{wt} \, u  + m} +   \sum_{m\ge n+1} u_{\mathrm{wt} \, u - m - 1}vx^{-\mathrm{wt} \, u  + m } \\
&\sim_n&
  \sum_{m=\mathrm{wt} \, u}^n u_{\mathrm{wt} \, u - m - 1} v x^{-\mathrm{wt} \, u  + m} \\
  & & \quad + \sum_{m\ge n+1} (-1)^{m+ \mathrm{wt} \, u} \sum_{j = 1}^{n+ \mathrm{wt} \, u} \binom{m-n-1}{j-1} \binom{m-n-j-1}{n+\mathrm{wt} \, u -j } u_{\mathrm{wt} \, u - n - j - 1} v x^{-\mathrm{wt} \, u  + m}  \\
&\sim_n&
 \sum_{k=1}^{n - \mathrm{wt} \, u + 1}  u_{-k} v x^{k-1} \\
  & & \quad +   \sum_{m\ge n+1} \sum_{k = n - \mathrm{wt} \, u + 2}^{2n + 1} (-1)^{m- \mathrm{wt} \, u}  \binom{m-n-1}{k-n + \mathrm{wt} \, u -2} \binom{m-k-\mathrm{wt} \, u}{2n-k + 1 } u_{-k} v x^{-\mathrm{wt} \, u  + m}  ,
\end{eqnarray*}
giving Eqn.\ (\ref{Y+}).  Eqn.\ (\ref{Y+-neg}) follows easily from the first two equalities in the calculation above and Eqn.\ (\ref{recursion-n-equation-lowest}).
\end{proof}

\begin{rem}{\em In practice, as in \cite{AB-details-of-Heisenberg}, one can often find a closed form for (\ref{Y+}) in terms of rational functions in $x$ which allows for ease of calculations.  Then using the $L(-1)$-derivative property, one can obtain similar expressions of equivalence modulo $O_n^\circ(V)$ for $Y^+ (u_{-k} \mathbf{1}, x)$,  in particular for $1 \leq k \leq 2n+1$.  As we show below in Proposition \ref{general-generators-prop}, 
 in many circumstances these $u_{-k}$, for $1 \leq k \leq 2n+1$, are enough to form a generating set for elements in $A_n(V)$ generated by $u$.
 }
\end{rem}

\subsection{Generators and relations for subalgebras of $A_n(V)$ arising from certain elements of $V$ using the reduction and recursion}\label{generators-section} 

In this section, we use this recursion relation  in $A_n(V)$ given in Proposition \ref{recursion-n}  to obtain results regarding the generators for certain subalgebras of $A_n(V)$ arising from elements in $V$ with certain nice properties.  Calculating a minimal set of generators for a given higher level Zhu algebra is in general quite difficult. The results that follow give some guidance in how to do this.

Recall (e.g. \cite{LL}) that for $V$ a vertex algebra, a subset $S \subset V$ is said to generate $V$, if $V$ is spanned by elements of the form
\[ V = \mathrm{span} \{u_{k_1}^{(1)} u_{k_2}^{(2)} \cdots u_{k_r}^{(r)} \mathbf{1} \; | \;  r \in \mathbb{N}, \  u^{(1)}, \dots, u^{(r)} \in S,  \ k_1, \dots, k_r \in \mathbb{Z} \},\]
and this is denoted by $V = \langle S \rangle$.   For $V$ a nontrivial vertex algebra, we say that $V$ is {\it positively generated at $K$ by $S$} for $K \in \mathbb{Z}$, if
\[ V = \mathrm{span} \{u_{k_1}^{(1)} u_{k_2}^{(2)} \cdots u_{k_r}^{(r)} \mathbf{1} \; | \;  r \in \mathbb{N}, \  u^{(1)}, \dots, u^{(r)} \in S,  \ k_1, \dots, k_r \leq -K \},\]
and denote this by $V = \langle S \rangle^K$.    If $S$ is finite, and $V$ is positively generated at $K = 1$ by $S$, then we say that $V$ is {\it strongly generated} by $S$.

Suppose $V = \langle S \rangle^K$.  For $v \in V$ and $r \in \N$, we define $F_r(v)$ to be the subspace of $V$ linearly spanned by all elements of the form
\begin{equation}\label{general-span} 
u_{-k_1}^{(1)} u_{-k_2}^{(2)} \cdots u_{-k_m}^{(m)} v ,
\end{equation}
for $k_1, \dots, k_m \geq K$ and $m\in\mathbb{N}$  with  $m \leq r$. Thus $F_r(v) \subseteq F_{r+1}(v)$, and $F_0(v) = \mathbb{C}v$.

We say that $V$ satisfies the {\it permutation property} if for all  $m \in \mathbb{Z}_+$, for $\sigma$ any permutation on $m$ letters,  for any  $u^{(1)}, u^{(2)}, \dots, u^{(m)}, v \in V$, and for any $k_1, \dots, k_m \in \mathbb{Z}$,  then
\begin{equation}\label{general-permutation}
u_{-k_1}^{(1)}\cdots u_{-k_m}^{(m)}v - u_{-k_{\sigma(1)}}^{(\sigma(1))} \cdots u_{-k_{\sigma(m)}}^{(\sigma(m))}v \in F_{m-1}(v).
\end{equation}
If Eqn.\ \eqref{general-permutation} holds only for some particular subspaces of $V$, say for $u^{(1)}, \dots, u^{(m)} \in U \subset V$ and $v \in W \subset V$, then we say that the {\it the permutation property holds for $U$  acting on $W$}.

If $V$ is a vertex algebra, $U = \langle S \rangle^1 \subset V$,  the permutation property holds for $U$ acting on  $W \subset V$, and for all $u \in U$, $w \in W$, we have $u_jw = 0$ if $j\geq0$, then we say that {\it $U$ acts strongly on $W$.}

For the following results, we let $V$ be a vertex operator algebra and for fixed $n \in \mathbb{N}$, we suppose there exists $u \in V$, such that $V^u =  \langle u \rangle^1$ is a subvertex algebra of $V$ generated by $u$ and $V^u$ satisfies the permutation property on $W \subset V$. Note that we do not assume that $V^u$ is necessarily a subvertex operator algebra.  We introduce the notation $A_n^u$ to denote the subalgebra of $A_n(V)$ generated by elements in $V^u$.
And we use the notation $A_n^u . W$ to denote the elements in $A_n(V)$ of the form $u_{-k_1} \cdots u_{-k_m}  v + O_n(V)$ for $v \in W$, and $k_1, \dots, k_m \in \mathbb{Z}_+$.  

\begin{rem} {\em Note that examples of vertex algebras which are strongly generated as $\langle u \rangle^1$ include the rank one Heisenberg vertex operator algebra and the vertex operator algebras arising from the Virasoro algebra.  Also note that in these cases, the vertex algebra is indeed a vertex operator algebra.  In particular, for any vertex operator algebra $V$ with conformal element $\omega$, we have the subvertex operator algebra of $V$ given by $\langle \omega \rangle^1$. }
\end{rem}

\begin{prop}\label{general-generators-prop}
Let $V$ be a vertex operator algebra.  Fix $n \in \mathbb{N}$, and suppose there is an element $u \in V$ with $\mathrm{wt} \, u \geq -n$,  and such that  the subvertex algebra $V^u = \langle u \rangle^1$ satisfies the permutation property.  Let $A_n^u$ be the subalgebra of $A_n(V)$ generated by elements in $V^u$.  Then elements in $A_n^u$ are generated by elements of the form 
\[ u_{-2n - 1}^{i_{2n+1}} u_{- 2n}^{i_{2 n}} \cdots u_{ - 1}^{i_1 }\mathbf{1} + O_n(V) ,\]
for $i_1, i_2, \dots, i_{2 n+1} \in \mathbb{N}$.  Moreover, if $V^u$ acts strongly on $W \subset V$, then any element in $A_n^u . W$ can be written, modulo $O_n^\circ(V)$, as a linear combination of such elements of the form $u_{-2n - 1}^{i_{2n+1}} u_{- 2n}^{i_{2 n}} \cdots u_{ - 1}^{i_1 }v + O_n(V)$.
\end{prop}

\begin{proof}
By Proposition \ref{recursion-n}, the fact that $V^u$ is spanned by elements of the form (\ref{general-span}) with $K = 1$, and $V^u$ acts strongly on $W$, we have that modulo $O^\circ_n(V)$, elements in $A_n^u . W$ can be written as a linear combination of elements of the form $u_{-k_1}\cdots u_{-k_r} v$ for $r \in \N$ and $k_j \in \{ 1, \dots, 2 n+1 \}$.  By induction on $r$, and the assumption that $V^u$ satisfies the permutation property (\ref{general-permutation}) on $W$,  the result follows.   
\end{proof}

In order to reduce the set of generators further, we observe the following:

\begin{lem}\label{strongly-acting-lemma}
 Let $V$ be a vertex operator algebra, and $U = \langle S \rangle^1 \subset V$, such that $U$ acts strongly on $W \subset V$.  If $v \in W$, then modulo $O^L(V)$, $U$ acts strongly on $v_{-k} {\bf 1}$ for all $k \in \mathbb{Z}_+$.
\end{lem}

\begin{proof}  We prove the result by induction on $k$.  The result is true for $k = 1$ since $v_{-1} {\bf 1} = v$.  Assume the result holds for $1\leq k' <k$.  For $i_1, \dots, i_m \in \mathbb{Z}_+$,  $u^{(1)}, \dots, u^{(m)} \in U$, and $w = u^{(1)}_{-i_1} \cdots u^{(m)}_{-i_m} v_{-(k-1)} {\bf 1}$, Lemma \ref{L-reduction-lemma} implies that
\begin{eqnarray*}
\lefteqn{ (k-1) u^{(1)}_{-i_1} u^{(2)}_{-i_2} \cdots u^{(m)}_{-i_m} v_{-k} {\bf 1} } \\
 &\approx& - (\mathrm{wt} \,  w ) u^{(1)}_{-i_1} \cdots u^{(m)}_{-i_m} v_{-(k-1)} {\bf 1} -  \sum_{j = 1}^m i_j u^{(1)}_{-i_1} \cdots u_{-i_{j-1}}^{(j-1)}u_{-i_{j-1}}  u^{(j)}_{-i_j - 1} u^{(j+1)} _{-i_{j+1} } \cdots u^{(m)}_{-i_m} v_{-(k-1)} {\bf 1} ,
\end{eqnarray*}
so that modulo $O^L(V)$, $u^{(1)}_{-i_1} u^{(2)}_{-i_2} \cdots u^{(m)}_{-i_m} v_{-k} {\bf 1}$ is a linear combination of elements comprised of $U$ acting strongly on $v_{-(k-1)}{\bf 1}$.  Therefore, modulo $O^L(V)$, $U$ acts strongly on $v_{-k} {\bf 1}$.  
\end{proof}

Now combining the restriction of the generating set given by Proposition \ref{general-generators-prop} and the additional relations given by $O^L(V)$, we can further restrict the generators as follows:

\begin{thm}\label{reducing-general-generators-thm}
Let $V$ be a vertex operator algebra.  Fix $n \in \mathbb{Z}_+$, and suppose there is a $u \in V$ with $\mathrm{wt} \, u \geq -n$,  and such that  $V^u = \langle u \rangle^1$ satisfies the permutation property and acts strongly on $W \subset V$.  Let $A_n^u$ be the subalgebra of $A_n(V)$ generated by elements in $V^u$.  Then elements in $A_n^u$ are generated by elements of the form 
\[ u_{-2n}^{i_{2n}} u_{-2 n + 1}^{i_{2n -1}} \cdots u_{ - 1}^{i_1 }\mathbf{1} + O_n(V),\]
for $i_1, i_2, \dots, i_{2n} \in \mathbb{N}$, and elements in $A_n^u . W$ can be expressed as linear combinations of elements of the form $u_{-2n}^{i_{2n}} u_{-2 n + 1}^{i_{2n -1}} \cdots u_{ - 1}^{i_1 }v_{-k} {\bf 1} + O_n(V)$ for $v \in W$ and $k \in \mathbb{Z}_+$.

Furthermore, 
\begin{equation}\label{-2n-1-reduction}
u_{-2n-1} v_{-k} {\bf 1} \approx -\left( \frac{\mathrm{wt} \, u + \mathrm{wt} \, v + 2n + k   -2 }{2n} \right)u_{-2n} v_{-k}  {\bf 1}  - \frac{k}{2n} u_{-2n} v_{-k-1} {\bf 1}. 
\end{equation} 
And for  $i_1, \dots, i_{2n} \in \mathbb{N}$ and $i_{2n+1}, k \in \mathbb{Z}_+$, we have 
\begin{eqnarray}\label{reduce-equation}
\lefteqn{\ \  \  \ \  \  \ \ u_{-2n-1}^{i_{2n + 1}} u_{-2n}^{i_{2n}} u_{-2n + 1}^{i_{2n-1}} \cdots u_{-1}^{i_1}  v_{-k} {\bf 1}  }\\
&\equiv_n & \! \! -\frac{1}{2n( i_{2n} + 1) } \Biggl( \biggl(  \mathrm{wt} \, v_{-k} {\bf 1}   - 1 +  \sum_{j = 1}^{2n+1} i_j (\mathrm{wt} \, u + j -1)  \biggr) u_{-2n - 1}^{i_{2n+1}-1} u_{-2 n}^{i_{2n}+1} u_{-2n+1}^{i_{2n-1}}  \cdots u_{ - 1}^{i_1 }  v_{-k} {\bf 1}  \nonumber \\
& &  - (2n+1) (i_{2n+1} -1)\biggl(  \sum_{j = 1}^{n+ \mathrm{wt} \, u} \! \binom{n + \mathrm{wt} \, u}{j-1}  u_{\mathrm{wt} \, u - n - j - 1} \biggr) u_{-2n-1}^{i_{2n+1} -2} u_{-2n}^{i_{2n}+1}u_{-2n+1}^{i_{2n-1}}  \cdots u_{-1}^{i_1}  v_{-k} {\bf 1}  \nonumber \\
&  &  + \, (2n-1) i_{2n-1} u_{-2n-1}^{i_{2n + 1}-1} u_{-2n}^{i_{2n} + 2}u_{-2n + 1}^{i_{2n-1} - 1} u_{-2n + 2}^{i_{2n-2}} \cdots u_{-1}^{i_1}  v_{-k} {\bf 1}  + \cdots  \nonumber  \\
&  &  \cdots + 2 i_2 u_{-2n - 1}^{i_{2n+1}-1} u_{-2 n}^{i_{2n}+1}u_{-2n+1}^{i_{2n-1}}  \cdots u_{-4}^{i_4}  u_{-3}^{i_3 + 1} u_{-2}^{i_2 - 1} u_{ - 1}^{i_1 } v_{-k} {\bf 1}  \nonumber \\
& &  + \, i_1 u_{-2n - 1}^{i_{2n+1}-1} u_{-2 n}^{i_{2n}+1} u_{-2n+1}^{i_{2n-1}} \cdots u_{-3}^{i_3} u_{-2}^{i_2 + 1} u_{ - 1}^{i_1-1 } v_{-k} {\bf 1}     + k u_{-2 n-1}^{i_{2n+1}- 1} u_{-2n}^{i_{2n}+1} u_{-2n +1}^{i_{2n-1}} \cdots u_{ - 1}^{i_1 }  v_{-k-1} {\bf 1} \nonumber \\
&  &   + f_{i_{2n+1} + i _{2n} + \cdots i_1 -1}( v_{-k} {\bf 1} ) \Biggr).\nonumber
\end{eqnarray}
for some $f_{i_{2n+1} + i _{2n} + \cdots i_1 -1}(  v_{-k} {\bf 1}  ) \in F_{i_{2n+1} + i _{2n} + \cdots i_1 -1}(v_{-k} {\bf 1} )$, where the term with the second summation in Eqn.\ \eqref{reduce-equation} is zero if $\mathrm{wt} \, u = -n$,  where if $n = 1$, this is to be interpreted as
\begin{eqnarray}\label{reduce-equation-n=1}
\lefteqn{u_{-3}^{i_{3}} u_{-2}^{i_{2}} u_{-1}^{i_1}  v_{-k} {\bf 1}  }\\
&\equiv_1 & \! \! -\frac{1}{2( i_{2} + 1) } \Biggl( \biggl(  \mathrm{wt} \, v_{-k} {\bf 1}   - 1 +  \sum_{j = 1}^{3} i_j (\mathrm{wt} \, u + j -1)  \biggr) u_{-3}^{i_{3}-1} u_{-2}^{i_{2}+1} u_{ - 1}^{i_1 }  v_{-k} {\bf 1}  \nonumber \\
& &  - 3 (i_{3} -1)\biggl(  \sum_{j = 1}^{1+ \mathrm{wt} \, u} \! \binom{1 + \mathrm{wt} \, u}{j-1}  u_{\mathrm{wt} \, u - j - 2} \biggr) u_{-3}^{i_{3} -2} u_{-2}^{i_{2}+1} u_{-1}^{i_1}  v_{-k} {\bf 1}  \nonumber \\
& &  + \, i_1 u_{-3}^{i_{3}-1} u_{-2}^{i_{2}+2} u_{ - 1}^{i_1-1 } v_{-k} {\bf 1}    + k  u_{-3}^{i_{3}- 1} u_{-2}^{i_{2}+1} u_{ - 1}^{i_1 }  v_{-k-1} {\bf 1} \nonumber \\
&  &   + f_{i_{3} + i _{2} +  i_1 -1}( v_{-k} {\bf 1} ) \Biggr),\nonumber
\end{eqnarray}
and where $f_{i_{2n+1} + i_{2n} + \cdots i_1 -1} (v_{-k} \mathbf{1}) = 0$ in both (\ref{reduce-equation}) and (\ref{reduce-equation-n=1}) if the $u_{-j}$'s for $j \in \mathbb{Z}_+$ commute. 
\end{thm}

\begin{proof} We first prove Eqns.\ (\ref{-2n-1-reduction}) and \eqref{reduce-equation}.   Eqn.\ (\ref{-2n-1-reduction}) follows from Eqn.\ (\ref{L-reduction-special-case}) applied to the case when  $j = 2n$.

Now note that for $u_{-2n -2} = u_{\mathrm{wt} \, u - m - 1}$  and  $m = \mathrm{wt} \, u + 2n + 1 \geq n + 1$, we have by Proposition \ref{recursion-n}, that for any $v \in V$,
\begin{eqnarray}\label{using-n-recursion}
u_{-2n-2} v &\sim_n &  - \sum_{j = 1}^{n+ \mathrm{wt} \, u} \binom{n + \mathrm{wt} \, u}{j-1} \binom{n + \mathrm{wt} \, u -j}{n+\mathrm{wt} \, u -j } u_{\mathrm{wt} \, u - n - j - 1} v\\
&=& - \sum_{j = 1}^{n+ \mathrm{wt} \, u} \binom{n + \mathrm{wt} \, u}{j-1}  u_{\mathrm{wt} \, u - n - j - 1} v, \nonumber
\end{eqnarray}
if $\mathrm{wt} \, u > -n$ and $u_{-2n-2}v \sim_n 0$ if $\mathrm{wt} \, u = -n$.  

Thus using (\ref{L(-1)-properties}), the permutation property of $V^u$ (\ref{general-permutation}), the fact that $V^u$ acts strongly on $W$, and (\ref{using-n-recursion}), we have that for $i_1, i_2, \dots, i_{2n+1} \in \mathbb{N}$, $k \in \mathbb{Z}_+$, and $\mathrm{wt} \, u > -n$
\begin{eqnarray*}
0 &\approx& (L(-1) + L(0)) u_{-2n - 1}^{i_{2n+1}} u_{-2 n}^{i_{2n}} \cdots u_{ - 1}^{i_1 } v_{-k} {\bf 1} \\
&=& \left(  \mathrm{wt} \, v_{-k} {\bf 1}  + \sum_{j = 1}^{2n+1} i_j (\mathrm{wt} \, u + j -1) \right) u_{-2n - 1}^{i_{2n+1}} u_{-2 n}^{i_{2n}} \cdots u_{ - 1}^{i_1 } v_{-k} {\bf 1} \\
& & \quad + \, (2n+1) i_{2n+1} u_{-2n -2} u_{-2n-1}^{i_{2n+1} -1} u_{-2n}^{i_{2n}} \cdots u_{-1}^{i_1}  v_{-k} {\bf 1}   \\
& & \quad + \, 2n i_{2n} u_{-2n-1}^{i_{2n + 1} + 1} u_{-2n}^{i_{2n} -1} u_{-2n + 1}^{i_{2n-1}} \cdots u_{-1}^{i_1}  v_{-k} {\bf 1}   + \cdots \\
& & \quad \cdots + 2 i_2 u_{-2n - 1}^{i_{2n+1}} u_{-2 n}^{i_{2n}} \cdots u_{-3}^{i_3 + 1} u_{-2}^{i_2 - 1} u_{ - 1}^{i_1 }  v_{-k} {\bf 1}  \\
& & \quad + \, i_1 u_{-2n - 1}^{i_{2n+1}} u_{-2 n}^{i_{2n}} \cdots u_{-2}^{i_2 + 1} u_{ - 1}^{i_1-1 }v_{-k} \mathbf{1} +k u_{-2 n-1}^{i_{2n+1}} u_{-2n}^{i_{2n}} \cdots u_{ - 1}^{i_1 }  v_{-k-1} {\bf 1} \\
& & \quad +  f_{i_{2n+1} + i _{2n} + \cdots i_1 -1}( v_{-k} {\bf 1} )\\
&\sim_n& \left( \mathrm{wt} \, v_{-k} {\bf 1}  +  \sum_{j = 1}^{2n+1} i_j (\mathrm{wt} \, u + j -1) \right) u_{-2n - 1}^{i_{2n+1}} u_{-2 n}^{i_{2n}} \cdots u_{ - 1}^{i_1 }  v_{-k} {\bf 1}  \\
& & \quad - \, (2n+1) i_{2n+1} \left(  \sum_{j = 1}^{n+ \mathrm{wt} \, u} \binom{n + \mathrm{wt} \, u}{j-1}  u_{\mathrm{wt} \, u - n - j - 1} \right) u_{-2n-1}^{i_{2n+1} -1} u_{-2n}^{i_{2n}} \cdots u_{-1}^{i_1}  v_{-k} {\bf 1}  \\
& & \quad + \, 2n i_{2n} u_{-2n-1}^{i_{2n + 1} + 1} u_{-2n}^{i_{2n} -1} u_{-2n + 1}^{i_{2n-1}} \cdots u_{-1}^{i_1}  v_{-k} {\bf 1}   \\
& & \quad + \,  (2n-1) i_{2n-1} u_{-2n-1}^{i_{2n + 1}} u_{-2n}^{i_{2n} + 1}u_{-2n + 1}^{i_{2n-1} - 1} u_{-2n + 2}^{i_{2n-2}} \cdots u_{-1}^{i_1} v_{-k} \mathbf{1} + \cdots \\
& & \quad \cdots +  2 i_2 u_{-2n - 1}^{i_{2n+1}} u_{-2 n}^{i_{2n}} \cdots u_{-3}^{i_3 + 1} u_{-2}^{i_2 - 1} u_{ - 1}^{i_1 } v_{-k} {\bf 1}  \\
& & \quad + \, i_1 u_{-2n - 1}^{i_{2n+1}} u_{-2 n}^{i_{2n}} \cdots u_{-2}^{i_2 + 1} u_{ - 1}^{i_1-1 }  v_{-k} \mathbf{1}  + k u_{-2 n-1}^{i_{2n+1}} u_{-2n}^{i_{2n}} \cdots u_{ - 1}^{i_1 }  v_{-k-1} {\bf 1} \\
& & \quad +  f_{i_{2n+1} + i _{2n} + \cdots i_1 -1}( v_{-k} {\bf 1} ),\\
\end{eqnarray*}
where $ f_{i_{2n+1} + i _{2n} + \cdots i_1 -1}( v_{-k} {\bf 1} ) \in F_{i_{2n+1} + i _{2n} + \cdots i_1 -1}(v_{-k} {\bf 1} )$, and this is zero if the $u_{-2n-1}, \dots, u_{-1}$ commute since these terms are only arising from reordering the $u_{-m}$'s after acting by $L(-1)$. 
Therefore if $i_{2n} \neq 0$, we have that 
\begin{eqnarray*} 
\lefteqn{ u_{-2n-1}^{i_{2n + 1} + 1} u_{-2n}^{i_{2n} -1} u_{-2n + 1}^{i_{2n-1}} \cdots u_{-1}^{i_1}   v_{-k} {\bf 1}   }\\
&\equiv_n& -\frac{1}{2n i_{2n} } \Biggl(  \left(  \mathrm{wt} \, v_{-k} {\bf 1}  + \sum_{j = 1}^{2n+1} i_j (\mathrm{wt} \, u + j -1) \right) u_{-2n - 1}^{i_{2n+1}} u_{-2 n}^{i_{2n}} \cdots u_{ - 1}^{i_1 }  v_{-k} {\bf 1}   \\
& & \quad - \, (2n+1) i_{2n+1} \left(  \sum_{j = 1}^{n+ \mathrm{wt} \, u} \binom{n + \mathrm{wt} \, u}{j-1}  u_{\mathrm{wt} \, u - n - j - 1} \right) u_{-2n-1}^{i_{2n+1} -1} u_{-2n}^{i_{2n}} \cdots u_{-1}^{i_1} v_{-k} {\bf 1}  \\
& & \quad  + \, (2n-1) i_{2n-1} u_{-2n-1}^{i_{2n + 1}} u_{-2n}^{i_{2n} + 1}u_{-2n + 1}^{i_{2n-1} - 1} u_{-2n + 2}^{i_{2n-2}} \cdots u_{-1}^{i_1}  v_{-k} {\bf 1}   + \cdots   \\
& & \quad \cdots +  2 i_2 u_{-2n - 1}^{i_{2n+1}} u_{-2 n}^{i_{2n}} \cdots u_{-4}^{i_4} u_{-3}^{i_3 + 1} u_{-2}^{i_2 - 1} u_{ - 1}^{i_1 } v_{-k} {\bf 1}  \\
& & \quad + \, i_1 u_{-2n - 1}^{i_{2n+1}} u_{-2 n}^{i_{2n}} \cdots u_{-3}^{i_3} u_{-2}^{i_2 + 1} u_{ - 1}^{i_1-1 } v_{-k} {\bf 1}  + k u_{-2 n-1}^{i_{2n+1}} u_{-2n}^{i_{2n}} \cdots u_{ - 1}^{i_1 }  v_{-k-1} {\bf 1} \\
& & \quad +  f_{i_{2n+1} + i _{2n} + \cdots i_1 -1}( v_{-k} {\bf 1} ) \Biggr).
\end{eqnarray*}
Equivalently, shifting  $i_{2n+1} \mapsto i_{2n + 1} -1$ and $i_{2n} -  1 \mapsto i_{2n}$, we have that for $i_{2n+1} \in \mathbb{Z}_+$ and $i_1, \dots, i_{2n} \in \mathbb{N}$, gives Eqn.\  \eqref{reduce-equation}. 

We are left to prove the statements in the first paragraph of the theorem. 

By Proposition \ref{general-generators-prop}, we have that $A_n^u$ is generated by elements of the form 
\[ u_{-2n - 1}^{i_{2n+1}} u_{- 2n}^{i_{2 n}} \cdots u_{ - 1}^{i_1 }\mathbf{1} + O_n(V) ,\]
for $i_1, i_2, \dots, i_{2 n+1} \in \mathbb{N}$, modulo $O_n^\circ(V) \subset O_n(V)$, and $A_n^u . W$ consists of linear combinations of elements of the form $u_{-2n - 1}^{i_{2n+1}} u_{- 2n}^{i_{2 n}} \cdots u_{ - 1}^{i_1 } v + O_n(V) $ for $v \in W$, also modulo $O_n^\circ(V) \subset O_n(V)$.

For fixed $v \in W$, and $i_1, i_2, \dots, i_{2n+1} \in \mathbb{N}$, we need to show that 
$u_{-2 n - 1}^{i_{2n +1}} u_{-2n}^{i_{2n}}  \cdots u_{ - 1}^{i_1 }v$ can be expressed as a linear combination of elements of the form $ u_{-2n}^{j_{2n}}  \cdots u_{ - 1}^{j_1 }v_{-k} {\bf 1}$, for $k \in \mathbb{Z}_+$ and $j_1, j_2, \dots, j_{2n} \in \mathbb{N}$.   We prove this by induction on the total degree of elements of the form $u_{-2 n - 1}^{i_{2n +1}} u_{-2n}^{i_{2n}}  \cdots u_{ - 1}^{i_1 }v_{-k} {\bf 1}$ where the total degree is defined as follows:  
For fixed $v \in W$, $k \in \mathbb{Z}_+$ and $i_1, i_2, \dots, i_{2n+1} \in \mathbb{N}$, we define the total degree of $u_{-2 n - 1}^{i_{2n +1}} u_{-2n}^{i_{2n}}  \cdots u_{ - 1}^{i_1 }v_{-k} {\bf 1}$  with respect to $v_{-k} {\bf 1}$ to be 
\[ \mathrm{deg} \, u_{-2 n - 1}^{i_{2n +1}}   u_{-2n}^{i_{2n}}  \cdots u_{ - 1}^{i_1 } v_{-k} {\bf 1} = \sum_{j=1}^{2n+1} i_j .\] 
Note that the degree is independent of $k$. 

For degree one elements acting on $v_{-k} {\bf 1}$, including $v_{-1} {\bf 1} = v$, if $i_{2n+1} = 0$, the result follows trivially, and if $i_{2n+1} = 1$, the result follows from Eqn.\ \eqref{-2n-1-reduction}.

Assume the result holds for elements of total degree less than or equal to $m$ acting on $v_{-k} {\bf 1}$ for $k \in \mathbb{Z}_+$, and let $i_{1}  + \cdots + i_{2n+1} = m+1$.    If $i_{2n+1} = 0$ then the result holds trivially.  If $i_{2n+1} >0$, the result follows by  Eqn.\  (\ref{reduce-equation}), by  induction on $i_{2n+1}$ at degree $m+1$,  by the inductive assumption for degree less than or equal to $m$, and by Lemma   \ref{strongly-acting-lemma}. This finishes the result in the case $\mathrm{wt} \, u > -n$.   In the case $\mathrm{wt} \, u = -n$, we simply note that since $u_{-2n-2}v \sim_n 0$, in this case the second term on the right hand side of Eqn.\ \eqref{reduce-equation} vanishes, and the results hold.    
\end{proof}

\subsection{Reducing generators of $A_n(V)$  further by establishing multiplication formulas}\label{mult-section}

We proceed to reduce the number of generators further by considering the types of elements produced from multiplying certain generators in $A_n(V)$.  We  first prove a general multiplication formula in the lemma below, and then a proposition employing one case of this formula to rewrite certain linear combinations of certain generators in $A_n(V)$ in terms of products of other generators.  This then ultimately allows us to further reduce the set of generators in Theorem \ref{last-generator-theorem}.

For the next lemma, we recall \cite{LL}, the normal ordering in $V$ given by 
\[ {}_{\circ}^{\circ} \,  u_j u_ k \,  {}_{\circ}^{\circ} \,  =  \left\{ \begin{array}{ll}
u_j u_k & \mbox{if $j<0$}\\
u_k u_j & \mbox{if $j \geq 0$} 
\end{array}
\right.  .\]

\begin{lem}\label{multiplication-lemma}
Let $V$ be a vertex operator algebra. Fix $n\in \mathbb{N}$. Let $t\in \mathbb{Z}_+$, and $i_1,\dots, i_t \in \mathbb{N}$.  For $1\leq s \leq t$, set $p_s=\sum_{j=1}^{s}i_j$, $p_0 = 0$,  and $r=\sum_{j=1}^{t}(\text{wt }u+j-1)i_j$.  We have the following formula for $u_{-t}^{i_{t}}\dots u_{-1}^{i_{1}}{\bf{1}}*_nv$, where $u, v\in V$

\begin{multline}\label{multiplication-formula-infinity} 
 u_{-t}^{i_{t}}\cdots u_{-1}^{i_{1}}{\bf{1}}*_nv\\
= \sum_{m=0}^n \sum_{j=-m}^n (-1)^m\binom{m+n}{n}\binom{n+r}{j+m}\displaystyle\sum_{\substack{k_1,\dots ,k_{p_t}\in \mathbb{N}\\k_1+\cdots +k_{p_t}=n-j}} \! \! {}_{\circ}^{\circ}P_{i_1,\dots, i_{t}}(k_1,\dots, k_{p_t}) {}_{\circ}^{\circ} v +g_{i_1,\dots, i_{t}}(v), 
\end{multline}
where 
\begin{equation}\label{define-P}
P_{i_1,\dots, i_{t}}(k_1,\dots, k_{p_t})=\prod_{s=1}^{t}\prod_{l=1}^{i_{s}}\binom{k_{l+p_{s-1}}+s-1}{s-1}u_{-k_{l+p_{s-1}}-s},
\end{equation}
and $g_{i_1,\dots, i_{t}}(v)$ is a linear combination of elements of the form  ${}_{\circ}^{\circ}u_{q_1}\cdots u_{q_{p_t}}{}_{\circ}^{\circ}v$, where $q_1, \dots,q_{p_t} \in \mathbb{Z}$ with at least one of the $q$'s nonnegative.

Furthermore if $n + r \geq 0$ and $r <0$, then the range for the sum in $j$ in Eqn.\  \eqref{multiplication-formula-infinity} can be considered to be  $-n\leq j \leq n + r$.
\end{lem}

\begin{proof}
From the definition of multiplication in the level $n$ Zhu algebra, we have
\begin{eqnarray*}
\lefteqn{  u_{-t}^{i_{t}}\cdots u_{-1}^{i_{1}}{\bf{1}}*_n v}\\
&=& \res_x\sum_{m=0}^n(-1)^m\binom{m+n}{n}x^{-m-n-1}\sum_{j=0}^{\infty}\binom{n+r}{j}x^j{}_{\circ}^{\circ}\prod_{s=1}^t\Big(\sum_{l \in \mathbb{Z}}\binom{l +s-1}{s-1}u_{-l-s}x^l\Big)^{i_s}{}_{\circ}^{\circ}v\\
&=& \sum_{m=0}^n\sum_{j=0}^{\infty}(-1)^m\binom{m+n}{n}\binom{n+r}{j}\displaystyle\sum_{\substack{k_1,\dots ,k_{p_t}\in \mathbb{Z}\\k_1+\cdots +k_{p_t}=n+m-j}}{}_{\circ}^{\circ}P_{i_1,\dots, i_{t}}(k_1,\dots, k_{p_t}){}_{\circ}^{\circ}v\\
&=&\sum_{m=0}^n\sum_{j=-m}^{\infty}(-1)^m\binom{m+n}{n}\binom{n+r}{j+m}\displaystyle\sum_{\substack{k_1,\dots ,k_{p_t}\in \mathbb{Z}\\k_1+\cdots +k_{p_t}=n-j}}{}_{\circ}^{\circ}P_{i_1,\dots, i_{t}}(k_1,\dots, k_{p_t}){}_{\circ}^{\circ}v,\
\end{eqnarray*}
where $P$ is given by Eqn.\ \eqref{define-P}.

The binomial coefficient $\binom{k_{l+p_{s-1}}+s-1}{s-1}$ appearing in the formula for $P_{i_1,\dots, i_{t}}(k_1,\dots, k_{p_t})$ is zero if $k_{l+p_{s-1}}$ is such that $-s< k_{j+p_{s-1}}<0$. Thus restricting the $k_{j+p_{s-1}}$'s to be nonnegative includes all the terms corresponding to products of $u_{i}$'s where all $i$'s are negative.  In other words, we have 
\begin{multline*}
u_{-t}^{i_{t}}\cdots u_{-1}^{i_{1}}{\bf{1}}*_nv \\
=\sum_{m=0}^n\sum_{j=-m}^{\infty}(-1)^m\binom{m+n}{n}\binom{n+r}{j+m}\displaystyle\sum_{\substack{k_1,\dots,k_{p_t}\in \mathbb{N}\\k_1+\cdots +k_{p_t}=n-j}} \! \! {}_{\circ}^{\circ} P_{i_1,\dots, i_{t}}(k_1,\dots, k_{p_t}) {}_{\circ}^{\circ} v+g_{i_1,\dots, i_{t}}(v),
\end{multline*}
where $g_{i_1,\dots, i_{t}}(v)$ is a linear combination of elements of the form  ${}_{\circ}^{\circ}u_{q_1}\cdots u_{q_{p_t}}{}_{\circ}^{\circ}v$, for $q_1, \dots, q_{p_t} \in \mathbb{Z}$ with at least one of the $q$'s  nonnegative.  Observing then that we must have $n-j= k_1 + \cdots + k_{p_t} \geq 0$ if $k_1, \dots, k_{p_t} \in \mathbb{N}$, we see that $j$ ranges only from $-m$ to $n$, proving Eqn.\  \eqref{multiplication-formula-infinity}.

If $n + r \geq 0$, then $\binom{n+r}{j+m}$ is zero when $j>n+r-m$, and thus if $r<0$, the expression becomes
\begin{eqnarray*}
\lefteqn{ u_{-t}^{i_{t}}\cdots u_{-1}^{i_{1}}{\bf{1}}*_n v}\\
&=& \sum_{m=0}^n\sum_{j=-m}^{n+r}(-1)^m\binom{m+n}{n}\binom{n+r}{j+m}\displaystyle\sum_{\substack{k_1,\dots ,k_{p_t}\in \mathbb{N}\\k_1+\cdots +k_{p_t}=n-j}}{}_{\circ}^{\circ}P_{i_1,\dots, i_{t}}(k_1,\dots, k_{p_t}){}_{\circ}^{\circ}v +g_{i_1,\dots, i_{t}}(v),
\end{eqnarray*}
proving the last statement.
\end{proof}

\begin{rem} {\em The formula above shows that the terms in the multiplication formula that are not part of $g_{i_1,\dots, i_{t}}(v)$ are indexed by ordered partitions of $2n$ and less. }
\end{rem}

\begin{rem}{\em 
$g_{i_1,\dots, i_{t}}(v)$ can be relatively manageable in the case that $V$ is, for example, the rank one Heisenberg vertex operator algebra, $M(1)$. For instance, in \cite{BVY-Heisenberg} and \cite{AB-details-of-Heisenberg}, respectively, where $A_1(M(1))$ and $A_2(M(1))$, respectively, are explicitly calculated, an ordering is defined on elements of $M(1)$ and Lemma \ref{multiplication-lemma} is used to induct on this ordering and reduce the number of required generators. A key part of this argument uses the fact that terms in $g_{i_1,\dots, i_{t}}(v)$ are of lower order than certain other terms in the corresponding multiplication formula.  In Proposition \ref{reduce-prop} and Theorem \ref{last-generator-theorem} below, we give an example of how this can be done.}
\end{rem}

\begin{rem}{\em In Proposition \ref{reduce-prop} and Theorem \ref{last-generator-theorem} below, we give just one example of how Eqn.\  \eqref{multiplication-formula-infinity} can be used to find further relations in $A_n(V)$ and to reduce the set of generators of $A_n(V)$.  However, one difficulty in using  Eqn.\  \eqref{multiplication-formula-infinity} in this way is that it is sometimes necessary to argue that certain coefficients are nonzero and this can be difficult, especially in cases where Proposition \ref{recursion-n} or Lemma \ref{multiplication-lemma} is being applied.  In particular, one often obtains terms whose coefficients are rather complicated combinations of binomial coefficients that must be simplified or at least shown to be nonzero. }
\end{rem}

\begin{cor}\label{multiplication-cor-special-case-more-general}
Let $V$ be a vertex operator algebra.  Fix $n \in \mathbb{N}$, and suppose there is a $u \in V$, such that  $\mathrm{wt}\,  u \geq -n$, and  $V^u = \langle u \rangle^1 \subset V$ acts strongly on $W \subset V$.   For $i,t \in \mathbb{Z}_+$ and $v \in W$, we have the following multiplication formula
\begin{multline}\label{multiplication-cor-formula}
u_{-t}^i \mathbf{1} *_n v =  \sum_{m=0}^n \sum_{j =0}^{m+n}   \sum_{\stackrel{p_1, p_2, \dots, p_i \in \mathbb{N}}{p_1 + p_2 + \cdots + p_i =  j} } (-1)^m \binom{m+n}{n}  \binom{n + i (\mathrm{wt}\, u + t - 1) }{m + n - j} \\
 \left( \prod_{l = 1}^i \binom{p_l +  t - 1}{t-1}  u_{-p_l -t} \right)  \, v  + g_{i}(v) , 
\end{multline}
where  $g_{i}(v)$ is a linear combination of elements of the form  ${}_{\circ}^{\circ}u_{q_1}\cdots u_{q_{i}} {}_{\circ}^{\circ}v$ where at least one of the  $q$'s is nonnegative.
That is, modulo $g_{i}(v)$, the product $u_{-t}^i \mathbf{1} *_n v$ involves a linear combination of terms of the form $u_{-p_1-t} u_{-p_2-t}  \cdots u_{-p_i-t}  \, v$ where $\mathbf{p} = (p_1, p_2, \dots, p_i)$ is an ordered partition of $j$ into $i$  nonnegative integer parts for $0 \leq j \leq m+ n \leq 2n$.  

Furthermore, when $i = 1$ and $u_{k-t}v = 0$ for $k\in \mathbb{Z}_+$, or if  $n + \wt u + t - 1 \geq 0$ and $\wt u + t <2$, then $g_1(v) = 0$.  In addition,  if $n \geq \wt u + t \geq 2$, then letting $k = \mathrm{wt} \, u + t$, we have 
\[g_1(v) =   \sum_{m=0}^{k - 2} \sum_{j=-1}^{m -k + 1}(-1)^m\binom{m+n}{n}\binom{n+ k -1 }{m + n - j }\binom{j  + t-1}{t-1} u_{-j  - t}  v , \]
and more generally
\[ g_1(v) =  \sum_{m=0}^n\sum_{j=-1}^{-\infty}(-1)^m\binom{m+n}{n}\binom{n+ \text{wt }u + t -1 }{m + n - j }\binom{j  + t-1}{t-1} u_{-j  - t}  v .\]
\end{cor}

\begin{proof} From Eqn.\  \eqref{multiplication-formula-infinity} applied in this case,  so that $i_j  = 0$ for $1\leq j <t$, $i_t = i$, $p_s=\sum_{j=1}^{s}i_j = 0$ if $0\leq s <t$ and $p_t = i$,   and $r=\sum_{j=1}^{t}(\text{wt }u+j-1)i_j =( \text{wt }u + t -1)i$, we have that for any $u,v \in V$, $t \in \mathbb{Z}_+$, and  $i \in \mathbb{N}$
\begin{multline*}
u_{-t}^{i}{\bf{1}}*_nv = \sum_{m=0}^n \sum_{j=-m}^n (-1)^m\binom{m+n}{n}\binom{n + i (\mathrm{wt}\, u + t - 1) }{j+m} \\
\displaystyle\sum_{\substack{k_1,\dots ,k_{i}\in \mathbb{N}\\k_1+\cdots +k_{i}=n-j}} \! \! {}_{\circ}^{\circ}
\prod_{l=1}^{i} \binom{ k_l  +  t - 1}{t-1} u_{-k_{l}-t} {}_{\circ}^{\circ} v +g_{i}(v), 
\end{multline*}
where  $g_{i}(v)$ is a linear combination of elements of the form ${}_{\circ}^{\circ}u_{q_1}\cdots u_{q_{i}} {}_{\circ}^{\circ}v$ where at least one of the  $q$'s is nonnegative, and the formula follows by shifting $j$ to  $n - j$.

To prove the last statement of the Corollary, we note that  if $i = 1$, then Eqn.\ (\ref{multiplication-cor-formula}) becomes 
\begin{eqnarray*}
u_{-t} \mathbf{1} *_n v &= &  \sum_{m=0}^n \sum_{j =0}^{m+n}   (-1)^m \binom{m+n}{n}  \binom{n + \mathrm{wt}\, u + t - 1 }{m + n - j}  \binom{j +  t - 1}{t-1}  u_{-j -t}   \, v 
 +  g_1(v) ,
\end{eqnarray*}
whereas from  the first equation in  the proof of Lemma \ref{multiplication-lemma}, or computing directly from the definition of $*_n$, we have 
\begin{eqnarray*}
\lefteqn{  u_{-t} {\bf{1}}*_n v}\\
&=& \sum_{m=0}^n\sum_{j=-m}^{\infty}(-1)^m\binom{m+n}{n}\binom{n+ \text{wt }u + t -1 }{j+m}\binom{n-j + t-1}{t-1} u_{j - n - t}  v\\
&=& \sum_{m=0}^n\sum_{j=m+n}^{-\infty}(-1)^m\binom{m+n}{n}\binom{n+ \text{wt }u + t -1 }{m + n - j }\binom{j  + t-1}{t-1} u_{-j  - t}  v \\
&=&   \sum_{m=0}^n \sum_{j =0}^{m+n}   (-1)^m \binom{m+n}{n}  \binom{n + \mathrm{wt}\, u + t - 1 }{m + n - j}  \binom{j +  t - 1}{t-1}  u_{-j -t}   \, v \\
& & \quad +  \sum_{m=0}^n\sum_{j=-1}^{-\infty}(-1)^m\binom{m+n}{n}\binom{n+ \text{wt }u + t -1 }{m + n - j }\binom{j  + t-1}{t-1} u_{-j  - t}  v ,
\end{eqnarray*}
where the last line is $g_1(v)$.  Thus if $u_{k-t}v = 0$ for $k\in \mathbb{Z}_+$, then $g_1(v) = 0$.  But in addition, the middle binomial term in $g_1(v)$ is zero if $n + \wt u + t - 1 \geq 0$ and $m+n - j > n + \wt u + t -1$.  Thus if $n + \wt u + t - 1 \geq 0$ and $\wt u + t <2$, it also follows that $g_1(v)  = 0$.  If $n \geq \wt u + t \geq 2$, then letting $k = \mathrm{wt} \, u + t$, we have 
\[g_1(v) =   \sum_{m=0}^{k - 2} \sum_{j=-1}^{m -k + 1}(-1)^m\binom{m+n}{n}\binom{n+ k -1 }{m + n - j }\binom{j  + t-1}{t-1} u_{-j  - t}  v . \]
This proves the last statement of the Corollary.
\end{proof}

Using this Corollary applied to the case that $v$ is a particular type of element in $A_n^u$, multiplied on the left by $u_{-1}^{i_1} {\bf 1}$, we have the following proposition:

\begin{prop}\label{reduce-prop}
Let $V$ be a vertex operator algebra.  Fix $n \in \mathbb{Z}_+$, and suppose there is a $u \in V$ with $\mathrm{wt} \, u \geq -n$, and such that  $V^u = \langle u \rangle^1$ satisfies the permutation property.  Then for $i_1, i_2, \dots, i_{2n} \in \mathbb{N}$, 
\begin{multline}\label{reduce1-prop}
 \binom{i_1 \mathrm{wt} \, u +n}{n} \sum_{m=0}^n (-1)^m \binom{i_1 \mathrm{wt} \, u}{m}   u_{-2n}^{i_{2n}} u_{-2 n + 1}^{i_{2n -1}} \cdots u_{ - 2}^{i_2 } u_{-1}^{i_1}  \mathbf{1} \\
\equiv_n u_{-1}^{i_1} \mathbf{1} *_n u_{-2n}^{i_{2n}} u_{-2 n + 1}^{i_{2n -1}} \cdots u_{ - 2}^{i_2 }\mathbf{1} + w',
\end{multline}
where $w'$ is a linear combination of terms of the form $u_{-2n}^{i_{2n}'} \cdots u_{-1}^{i_1'} \mathbf{1}$ where either $ i_1' + \cdots + i_{2n}' < i_1 + \cdots + i_{2n}$, or $ i_1' + \cdots + i_{2n}' = i_1 + \cdots + i_{2n}$ and  $i_1'< i_1$.    

Furthermore, the coefficient of the term on the lefthand side of Eqn.\ (\ref{reduce1-prop}) is zero if and only if $0< | i_1 \mathrm{wt} \, u | \leq n$, and is equal to 
\begin{equation}\label{reduce2-prop}
(-1)^n  \binom{i_1 \mathrm{wt} \, u +n}{n} \binom{i_1 \mathrm{wt} \, u  -1}{n}
\end{equation} 
if $i_1 \mathrm{wt} \, u >n$.
\end{prop}

\begin{proof}
We apply Corollary  \ref{multiplication-cor-special-case-more-general} to the case when 
$v = u_{-2n}^{i_{2n}} u_{-2 n + 1}^{i_{2n -1}} \cdots u_{ - 2}^{i_2 }\mathbf{1}  \in  F_{i_2 + \cdots i_{2n}} ({\bf 1}) \subset V^u$,  $t = 1$, and $i = i_1$. Thus Eqn.\  \eqref{multiplication-cor-formula} holds with $g_{i_1} (v) \in F_{i_1 + i_2 + \cdots i_{2n} -1} ( {\bf 1})$, and the only term in $u_{-1}^{i_1} \mathbf{1} *_n v$ that is not in  $F_{i_1 + \cdots + i_{2n}  -1}({\bf 1})$ and involves $u_{-2n}^{j_{2n}} u_{-2 n + 1}^{j_{2n -1}} \cdots u_{ - 1}^{j_1}\mathbf{1}$ with $j_1 \geq i_1$ is the one corresponding to the partition of $j = 0$ given by $\mathbf{p} = (0,0,\dots, 0)$.  That is, we have

\begin{eqnarray*}
\lefteqn{u_{-1}^{i_1} \mathbf{1} *_n u_{-2n}^{i_{2n}} u_{-2 n + 1}^{i_{2n -1}} \cdots u_{ - 2}^{i_2 }\mathbf{1}}\\
&=&  \sum_{m=0}^n  (-1)^m \binom{m+n}{n}  \binom{n + i_1 \mathrm{wt}\, u }{m + n }  u_{-1}^{i_1}  \, ( u_{-2n}^{i_{2n}} u_{-2 n + 1}^{i_{2n -1}} \cdots u_{ - 2}^{i_2 }\mathbf{1})  \\
& & \quad +  \sum_{m=0}^n \sum_{j =1}^{m+n}  \ \ \sum_{\stackrel{p_1, p_2, \dots, p_{i_1} \in \mathbb{N}}{p_1 + p_2 + \cdots + p_{i_1} =  j} } (-1)^m \binom{m+n}{n}  \binom{n + i_1 \mathrm{wt}\, u}{m + n - j}  u_{-p_1-1} u_{-p_2-1}  \cdots \\
& & \quad \cdots u_{-p_{i_1}-1}  \, ( u_{-2n}^{i_{2n}} u_{-2 n + 1}^{i_{2n -1}} \cdots u_{ - 2}^{i_2 }\mathbf{1}) + f_{i_1 + \cdots i_{2n} -1}({\bf 1}) \\
&=&  \sum_{m=0}^n  (-1)^m \binom{m+n}{n}  \binom{n + i_1 \mathrm{wt}\, u}{m + n }  u_{-2n}^{i_{2n}} u_{-2 n + 1}^{i_{2n -1}} \cdots u_{ - 2}^{i_2 } u_{-1}^{i_1}  \mathbf{1} +  f'_{i_1 + i_2 + \cdots + i_{2n} - 1}( \bf 1) \\
& & \quad +  \sum_{m=0}^n \sum_{j =1}^{m+n}  \ \ \sum_{\stackrel{p_1, p_2, \dots, p_{i_1} \in \mathbb{N}}{p_1 + p_2 + \cdots + p_{i_1} =  j} } (-1)^m \binom{m+n}{n}  \binom{n + i_1 \mathrm{wt}\, u }{m + n - j}  u_{-p_1-1} u_{-p_2-1}  \cdots \\
& & \quad \cdots u_{-p_{i_1}-1}  \, ( u_{-2n}^{i_{2n}} u_{-2 n + 1}^{i_{2n -1}} \cdots u_{ - 2}^{i_2 }\mathbf{1}) +   f_{i_1 + i_2 + \cdots + i_{2n} - 1}( \bf 1)
\end{eqnarray*}
for some  $f_{i_1+ \cdots + i_{2n} -1}({\bf 1}),  f'_{i_1 + i_2 + \cdots + i_{2n} - 1}( \bf 1) \in F_{i_1 + \cdots + i_{2n} -1}({\bf 1})$.  Therefore
\begin{eqnarray*}
\lefteqn{\sum_{m=0}^n  (-1)^m \binom{m+n}{n}  \binom{n+ i_1 \mathrm{wt}\, u}{m + n }  u_{-2n}^{i_{2n}} u_{-2 n + 1}^{i_{2n -1}} \cdots u_{ - 2}^{i_2 } u_{-1}^{i_1}  \mathbf{1} }\\
&=& u_{-1}^{i_1} \mathbf{1} *_n u_{-2n}^{i_{2n}} u_{-2 n + 1}^{i_{2n -1}} \cdots u_{ - 2}^{i_2 }\mathbf{1} \\
& & \quad - \sum_{m=0}^n \sum_{j =1}^{m+n}  \ \ \sum_{\stackrel{p_1, p_2, \dots, p_{i_1} \in \mathbb{N}}{p_1 + p_2 + \cdots + p_{i_1} =  j} } (-1)^m \binom{m+n}{n}  \binom{n+ i_1 \mathrm{wt}\, u}{m + n - j}  u_{-p_1-1} u_{-p_2-1}  \cdots \\
& & \quad \cdots u_{-p_{i_1}-1}  \, ( u_{-2n}^{i_{2n}} u_{-2 n + 1}^{i_{2n -1}} \cdots u_{ - 2}^{i_2 }\mathbf{1}) + h_{i_1+ \cdots + i_{2n}-1}({\bf 1}),\\
\end{eqnarray*}  
where $h_{i_1+ \cdots + i_{2n}-1}({\bf 1}) \in F_{i_1+ \cdots + i_{2n}-1}({\bf 1})$.

Using the permutation property on the terms in the sum involving ordered partitions, and breaking out the partitions in the sum over ordered partitions that involve a $u_{-2n-1}$ term (namely those with $m = n$, $j = 2n$, and $\mathbf{p} = (p_1, p_2, \dots, p_{i_1})$ consisting of $i_1-1$ zeros, and one $2n$), we have
\begin{multline*}
\sum_{m=0}^n  (-1)^m \binom{m+n}{n}  \binom{n+ i_1 \mathrm{wt}\, u}{m + n }  u_{-2n}^{i_{2n}} u_{-2 n + 1}^{i_{2n -1}} \cdots u_{ - 2}^{i_2 } u_{-1}^{i_1}  \mathbf{1} \\
= u_{-1}^{i_1} \mathbf{1} *_n u_{-2n}^{i_{2n}} u_{-2 n + 1}^{i_{2n -1}} \cdots u_{ - 2}^{i_2 }\mathbf{1} - (-1)^n  i_1 \frac{(2n)!}{n!n!} u_{-2n-1} u_{-2n}^{i_{2n}} u_{-2 n + 1}^{i_{2n -1}} \cdots u_{ - 2}^{i_2 } u_{-1}^{i_1 - 1} \mathbf{1}  \\
+ w + h'_{i_1 + i_2 + \cdots  + i_{2n} - 1}(\mathbf{1}) ,
\end{multline*}
for some  $h'_{i_1 + i_2 + \cdots  + i_{2n} - 1}(\mathbf{1}) \in F_{i_1 + i_2 + \cdots + i_{2n} - 1}(\mathbf{1})$, and 
where $w$ is a linear combination of terms of the form $u_{-2n}^{i_{2n}'} \cdots u_{-1}^{i_1'} \mathbf{1}$ where $ i_1' + \cdots + i_{2n}' \leq i_1 + \cdots + i_{2n}$ but  $i_1'< i_1$.

Since $\mathrm{wt} \, u \geq -n$, we can use Proposition \ref{recursion-n} to replace any terms which occur in $h'_{i_1+ \cdots + i_{2n}-1}({\bf 1})$ involving $u_{-k}$ for $k \geq 2n+2$ with terms involving sums of $u_{-j}$ for $1 \leq j \leq 2n+1$, and thus replace $h'_{i_1+ \cdots + i_{2n}-1}({\bf 1})$ with $h''_{i_1+ \cdots + i_{2n}-1}({\bf 1}) \in F_{i_1+ \cdots + i_{2n}-1}({\bf 1})$ which only involves $u_{-k}$ for $1 \leq k \leq 2n + 1$.

Using Theorem \ref{reducing-general-generators-thm}, and in particular repeatedly using Eqn.\  (\ref{reduce-equation}),  we can rewrite the terms involving $u_{-2n-1}$ as linear combinations of terms involving only $u_{-k}$ for $1\leq k \leq 2n$ while not introducing any higher powers of $u_{-1}$ except for possibly in lower degree terms, meaning terms in $F_r (\mathbf{1})$ for $r< i_1 + \cdots + i_{2n}$.   Therefore we have
\begin{multline}\label{reduce1}
\sum_{m=0}^n  (-1)^m \binom{m+n}{n}  \binom{n + i_1 \mathrm{wt}\, u}{m + n }  u_{-2n}^{i_{2n}} u_{-2 n + 1}^{i_{2n -1}} \cdots u_{ - 2}^{i_2 } u_{-1}^{i_1}  \mathbf{1} \\
\equiv_n u_{-1}^{i_1} \mathbf{1} *_n u_{-2n}^{i_{2n}} u_{-2 n + 1}^{i_{2n -1}} \cdots u_{ - 2}^{i_2 }\mathbf{1} + w',
\end{multline}
where $w'$ is a linear combination of terms of the form $u_{-2n}^{i_{2n}'} \cdots u_{-1}^{i_1'} \mathbf{1}$ where either $ i_1' + \cdots + i_{2n}' < i_1 + \cdots + i_{2n}$, or $ i_1' + \cdots + i_{2n}' = i_1 + \cdots + i_{2n}$ and  $i_1'< i_1$. 

To prove that the coefficient on the  lefthand side of Eqn.\ (\ref{reduce1}) is that given as in Eqn.\  (\ref{reduce1-prop}), we let $j = i_1 \mathrm{wt} \, u$ and observe that
\begin{eqnarray}\label{coefficient-eqn}
C_j &=& \sum_{m=0}^n  (-1)^m \binom{m+n}{n}  \binom{j + n}{m + n }  =   \sum_{m=0}^n (-1)^m \frac{(m+n)!}{n!m!} \frac{(j + n)!}{(m+n)!(j-m)!} \\
&=&  \sum_{m=0}^n (-1)^m \frac{(j+n)!}{n!m!(j-m)!} =  \frac{(j+n)!}{n!j!} \sum_{m=0}^n (-1)^m \frac{j!}{m!(j-m)!} \nonumber \\
&=& \binom{j+n}{n} \sum_{m=0}^n (-1)^m \binom{j}{m} \nonumber .
\end{eqnarray}

Thus, if  $0< j \leq n$
\begin{eqnarray*}
C_j = \binom{j+n}{n}    \sum_{m=0}^n (-1)^m \binom{j}{m} x^m \Big|_{x = 1} = \binom{j+n}{n}  (1-x)^j |_{x=1} = 0.
\end{eqnarray*}
If $-n\leq j  <0$, $\binom{j+n}{n} = 0$, and thus $C_j = 0$.   
If $j >n$, then it follows from
\[ \binom{j}{m} = \binom{j-1}{m} + \binom{j-1}{m-1}\]
and induction on $j$ that 
\[ \sum_{m=0}^n (-1)^m \binom{j}{m} = (-1)^n \binom{j-1}{n} \]
which gives 
\[
C_j  =  (-1)^n \binom{j+n}{n}  \binom{j-1}{n} \neq 0 \]
in this case, proving Eqn.\ (\ref{reduce2-prop}).

Finally if   $j<-n$, then using the fact that $\binom{j}{m} = (-1)^m \binom{-j+m-1}{m}$, we have 
\begin{eqnarray*}
C_j &=& \binom{j+n}{n}  \sum_{m=0}^n (-1)^m \binom{j}{m} \ = \  \binom{j+n}{n}  \sum_{m=0}^n  \binom{-j+m -1}{m}
\end{eqnarray*} 
which is always nonzero since $\binom{j+ n}{n} \neq 0$, and the summation consists only of positive terms when  $j<-n$, proving the last statement of the Proposition. 
\end{proof}

Proposition \ref{reduce-prop} allows us to reduce the generators further.  In particular, we  have the following refinement of Theorem \ref{reducing-general-generators-thm}:

\begin{thm}\label{last-generator-theorem}
Let $V$ be a vertex operator algebra.  Fix $n \in \mathbb{Z}_+$, and suppose there is an element $u \in V$ with $\mathrm{wt} \, u \geq -n$,  and such that  the vertex subalgebra $V^u = \langle u \rangle^1 \subset V$ satisfies the permutation property.  Let $A_n^u$ be the subalgebra of $A_n(V)$ generated by elements in $V^u$.  Then $A_n^u$ is generated by elements of the form 
\[ u_{-2n}^{i_{2n}} u_{-2 n + 1}^{i_{2n -1}} \cdots u_{ - 1}^{i_1}\mathbf{1}+ O_n(V) ,  \mbox{and} \ u_{-1}^{j} \mathbf{1} + O_n(V),\]
for $i_1, i_2, \dots, i_{2n} \in \mathbb{N}$, with  $-n \leq  i_1 \mathrm{wt} \, u   \leq n$, and $j \in \mathbb{Z}_+$.  
\end{thm}

\begin{proof}  By Theorem \ref{reducing-general-generators-thm}, $A_n^u$ is generated by $u^{i_{2n}}_{-2n} u^{i_{2n-1}}_{-2n+1} \cdots u_{-1}^{i_1} \mathbf{1} + O_n(V)$ for $i_1, i_2, \dots, i_{2n} \in \mathbb{N}$.  We proceed by induction on the following lexicographical ordering  of such generators to show that each of these elements $u^{i_{2n}}_{-2n} u^{i_{2n-1}}_{-2n+1} \cdots u_{-1}^{i_1} \mathbf{1} + O_n(V)$ in $A_n^u$ can be generated by the elements prescribed in the statement of the  present theorem.  The lexicographical ordering is as follows:
\begin{multline}\label{ordering}
u^{i_{2n}'}_{-2n} u^{i_{2n-1}'}_{-2n+1} \cdots u_{-1}^{i_1'} \mathbf{1} + O_n(V) <  u^{i_{2n}}_{-2n} u^{i_{2n-1}}_{-2n+1} \cdots u_{-1}^{i_1} \mathbf{1} + O_n(V) \\
 \iff \  \left\{ \begin{array}{l} 
i_{2n}' + i_{2n-1}' + \cdots +  i_1' <  i_{2n} + i_{2n-1} + \cdots + i_1 \\
 \mbox{or}\\ 
i_{2n}' + i_{2n-1}' + \cdots + i_1' =  i_{2n} + i_{2n-1} + \cdots + i_1
\ \mbox{and $i_1'<i_1$.}
\end{array} \right.
\end{multline}

First note that the theorem immediately holds for elements $ u^{i_{2n}}_{-2n} u^{i_{2n-1}}_{-2n+1} \cdots u_{-1}^{i_1} \mathbf{1} + O_n(V) \in A_n^u \subset A_n(V)$ of total degree $i_{2n} + i_{2n-1} + \cdots + i_1 = 0$ or $1$, as well as for any total degree as long as $i_1 = 0$, or $i_{2n} = i_{2n-1} = \cdots = i_2 = 0$.

We make the inductive assumption that the result holds for $u_{-2n}^{i_{2n}'} \cdots  u_{-1}^{i_1'}  \mathbf{1} + O_n(V)$ of total degree $i_{2n}' + \cdots + i_1'$, and we proceed to prove the statement holds for $u_{-2n}^{i_{2n}} \cdots  u_{-1}^{i_1}  \mathbf{1} + O_n(V)$ of total degree  $i_{2n} + \cdots + i_1 = i_{2n}' + \cdots + i_1' + 1$.  By Proposition  \ref{reduce-prop}, the coefficient on the lefthand side of Eqn.\ (\ref{reduce1-prop}) is zero if and only if $0 < | i_1 \mathrm{wt} \, u | \leq n$.  Thus as long as either $i_1 \mathrm{wt} \, u > n$ or 
 $i_1 \mathrm{wt} \, u <-n$, we have
\begin{multline}\label{reduce2}
u_{-2n}^{i_{2n}} u_{-2 n + 1}^{i_{2n -1}} \cdots u_{ - 2}^{i_2 } u_{-1}^{i_1}  \mathbf{1} \\
\equiv_n \left(  \binom{i_1 \mathrm{wt} \, u +n}{n} \sum_{m=0}^n (-1)^m \binom{i_1 \mathrm{wt} \, u}{m}  \right)^{-1} u_{-1}^{i_1} \mathbf{1} *_n u_{-2n}^{i_{2n}} u_{-2 n + 1}^{i_{2n -1}} \cdots u_{ - 2}^{i_2 }\mathbf{1} + w',
\end{multline}
where $w'$ is a linear combination of terms involving only elements of the form $u_{-2n}^{i_{2n}'} \cdots u_{-1}^{i_1'} \mathbf{1}$ where either $ i_1' + \cdots + i_{2n}' < i_1 + \cdots + i_{2n}$, or $ i_1' + \cdots + i_{2n}' = i_1 + \cdots + i_{2n}$ and  $i_1'< i_1$.

 By the inductive assumption, the terms in (\ref{reduce2}) of total degree less than $i_{2n} +  \cdots + i_1$ can be expressed in terms of the generators specified in the statement of the theorem.  For the terms of equal total degree but with $i_1'<i_1$, we induct on $i_1'$.  That is, we must show that the result holds for terms of the form $u_{-2n}^{i_{2n}'} \cdots u_{-1}^{i_1'} \mathbf{1}$ where $ i_1' + \cdots + i_{2n}' = i_1 + \cdots + i_{2n}$ and $i_1' <i_1$.  
  
We consider the  two cases: Case I.    $-n \leq i_1' \mathrm{wt} \, u \leq n$; Case II.  $i_1'  >  | \frac{n}{\mathrm{wt} \, u}|$.

The result holds if  $i_1' \leq  | \frac{n}{\mathrm{wt} \, u}|$, since then  $-n\leq i_1' \mathrm{wt} \, u \leq n$, and so $u_{-2n}^{i_{2n}'} \cdots u_{-1}^{i_1'} \mathbf{1}$ is in the list of generators specified by the theorem.  This includes the base case for our induction on $i_1'$.  We make the inductive assumption that the result holds for $i_1'$ and prove the result for $i_1'+ 1  >  | \frac{n}{\mathrm{wt} \, u}|$. In this case, by  Proposition \ref{reduce-prop},  we have 
\begin{multline}\label{reduce3}
u_{-2n}^{i_{2n}'} u_{-2 n + 1}^{i_{2n -1}'} \cdots u_{ - 2}^{i_2' } u_{-1}^{i_1' + 1}  \mathbf{1} \\
\equiv_n \left( \binom{(i_1'+1) \mathrm{wt} \, u +n}{n} \sum_{m=0}^n (-1)^m \binom{(i_1'+1) \mathrm{wt} \, u}{m}  \right)^{-1} u_{-1}^{i_1' + 1} \mathbf{1} *_n u_{-2n}^{i_{2n}'} u_{-2 n + 1}^{i_{2n -1}'} \cdots u_{ - 2}^{i_2' }\mathbf{1} + w'',
\end{multline}
where $w''$ is a linear combination of terms of the form $u_{-2n}^{i_{2n}''} \cdots u_{-1}^{i_1''} \mathbf{1}$ where either $ i_1'' + \cdots + i_{2n}'' < ( i_1' + 1) + i_2' + \cdots + i_{2n}'$, or $ i_1'' + \cdots + i_{2n}'' = (i_1' + 1) + i_2'  + \cdots + i_{2n}'$ and  $i_1''< i_1'$.  By the inductive assumptions on the total degree, and on $i_1'$, the result follows.  
\end{proof}

\section{On the definition of the higher level Zhu algebras and applications}\label{definition-clarification-section}

In this section, we prove a result that helps motivate the definition of $O_n(V)$ when $n \in \mathbb{Z}_+$, namely that one must define $O_n(V)$ to include all elements of the form $(L(-1) + L(0))v$ for $v \in V$.  In addition, as we note in Remark \ref{OL-remark} below, this theorem also singles out some of the relations that must hold in $A_n(V)$ that are contained in $O^L(V)$, but not necessarily in $O_n^\circ(V)$.  Thus the theorem is useful in determining the structure of $A_n(V)$.  

\begin{thm}\label{clarify-definition-thm}
Suppose that $V$ is a vertex operator algebra that is strongly generated by $u \in V$.  Then  
\[O_n^L(V)\not\subset O_n^\circ(V)\]
as long as $n \in \mathbb{Z}_+$ and one of the following hold:

(i) $1\leq \mathrm{wt} \, u \leq n$;

or 

(ii) $\mathrm{wt} \, u = 0$, and  $u_{-2} \mathbf{1}  \neq 0$.
\end{thm}

\begin{proof} Note that if $V$ is a vertex operator algebra, then for any $u \in V$, we have that $(L(0)+L(-1)) u = (L(0)+L(-1)) u_{-1}{\bf{1}}= (\mathrm{wt} \, u) u_{-1}{\bf{1}}+ u_{-2}{\bf{1}}$.  Thus if $\mathrm{wt} \, u \neq 0$ then $(L(0)+L(-1)) u$ is a linear combination of vectors in $V$ that contains a scalar multiple of $u=u_{-1}{\bf{1}}$.   We will show that for all $n>0$, if one of the two cases, (i) or (ii) hold, then the vector $(\mathrm{wt} \, u) u_{-1}{\bf{1}}+ u_{-2}{\bf{1}} \notin O_n^\circ(V)$, by showing that for any $v,w \in V$ the expression $v \circ_n w$ written as a linear combination of vectors in $V$ either cannot contain a nonzero scalar multiple of $u = u_{-1}{\bf{1}}$ or cannot contain a nonzero scalar multiple of $u_{-2} \mathbf{1}$.

Since $V$ is strongly generated by $u$, any $v \in V$ can be written as a linear combination of elements of the form  $u_{-k_1} u_{-k_2 } \cdots u_{ - k_m} \mathbf{1}$, for $m \in \mathbb{N}$, and $k_1, \dots, k_m  \in \mathbb{Z}_+$.  Thus it is sufficient to consider $v \circ_n w$ of the form $u_{-k_1} u_{-k_2 } \cdots u_{ - k_m} \mathbf{1} \circ_n w$.   

If $m = 0$, we have $\mathbf{1} \circ_n w =  \res_x\frac{(1+x)^{n}}{x^{2n+2}} w = 0$.  If $m >0$, we
let 
\begin{equation}\label{define-m}
t = \sum_{r=1}^{m}(\mathrm{wt} \, u + k_r  -1),  \quad \mbox{and} \quad s = \sum_{r = 1}^{m} k_r , \quad \mbox{and note that $t = (\mathrm{wt} \, u -1)m + s$.}
 \end{equation}
If $\mathrm{wt} \, u \geq 0$, then $\mathrm{wt} \, (u_{-k_1} u_{-k_2 } \cdots u_{ - k_m} \mathbf{1} )= t \geq 0$, and we have
\begin{eqnarray*}
\lefteqn{ u_{-k_1} u_{-k_2 } \cdots u_{ - k_m} \mathbf{1} \circ_n w}\\
&= & \res_x\frac{(1+x)^{t+n}}{x^{2n+2}} \, {}_{\circ}^{\circ} \, (Y(u_{-k_1}{\bf{1}}, x)) \cdots (Y(u_{-k_m}{\bf{1}}, x)) \, {}_{\circ}^{\circ} \, w\\
&=& \res_x \sum_{l = 0}^{t+n} \binom{t +n }{l}x^{l-2n-2} \, {}_{\circ}^{\circ} \, \left( \frac{1}{(k_1-1)!}  \left(\frac{d}{dx}\right)^{k_1 -1}\sum_{j \in \mathbb{Z}} u_{-j}x^{j-1} \right)\cdot\\
& & \quad  \left( \frac{1}{(k_2-1)!}  \left(\frac{d}{dx}\right)^{k_2-1}\sum_{j \in \mathbb{Z}} u_{-j}x^{j-1} \right) \cdots
 \left( \frac{1}{(k_m - 1)!}  \left(\frac{d}{dx}\right)^{k_m - 1} \sum_{j \in \mathbb{Z}}  u_{-j}x^{j-1}\right)   \, {}_{\circ}^{\circ} \, w\\
 &=& \res_x \sum_{l =0 }^{t+n} \binom{t+n }{l}x^{l-2n-2} \, {}_{\circ}^{\circ} \, \left(\sum_{j\in \mathbb{Z}}\binom{j-1}{k_1-1} u_{-j}x^{j-k_1}\right)  \left(  \sum_{j \in \mathbb{Z}} \binom{j-1}{k_2 - 1}  u_{-j}x^{j-k_2} \right) \cdot \\
& & 
 \cdots \left( \sum_{j \in \mathbb{Z}}  \binom{j-1}{k_m - 1} u_{-j}x^{j-k_m} \right) \, {}_{\circ}^{\circ} \, w\\
 &=& \res_x \sum_{l = 0}^{t+n}  \binom{t+n  }{l}x^{l-2n-2-s} \,  {}_{\circ}^{\circ} \, \left(\sum_{j\in \mathbb{Z}}\binom{j-1}{k_1 - 1} u_{-j}x^{j}\right) \left( \sum_{j \in \mathbb{Z}} \binom{j-1}{k_2 - 1} u_{-j}x^{j} \right) \cdot \\
& &  \cdots \left(  \sum_{j \in \mathbb{Z}} \binom{j-1}{k_m -1 } u_{-j}x^{j} \right) \, {}_{\circ}^{\circ} \, w.
 \end{eqnarray*}
 
Each term in this expression is a scalar multiple of something of the form 
\begin{equation}\label{terms}
{}_{\circ}^{\circ} \, u_{j_1}\cdots u_{j_m}\, {}_{\circ}^{\circ} \, w.
\end{equation}

After taking the residue above, the only terms appearing in $u_{-k_1} u_{-k_2} \cdots u_{ - k_m} \mathbf{1} \circ_n w$ are ones of this form (\ref{terms}) where 
\begin{equation*}
-j_1-\cdots-j_m +l-2n-2- s =-1
\end{equation*} 
or equivalently,
\begin{equation}\label{equation-for-sum}
j_1+\cdots+j_m =l-2n-1-s
\end{equation} 
with $0\leq l \leq n+t$, and $t$ and $s$ as in (\ref{define-m}).  That is, we have
\begin{equation}\label{equation-for-sum2}
-2n-1-s  \leq   j_1+\cdots+j_m  \leq t-n-1-s = (\mathrm{wt} \, u - 1)m -n - 1 .
\end{equation} 

Since $w \in V = \langle u \rangle^1$, we have that $w$ is a linear combination of elements of the form $u_{-p_1} \cdots u_{-p_q} \mathbf{1}$, 
for some $p_1, \dots, p_q \in \mathbb{Z}_+$ and $q \in \mathbb{N}$, and 
\[ \mathrm{wt} \, u_{-p_1} \cdots u_{-p_q} \mathbf{1} = q \, \mathrm{wt} \, u + p_1 + \cdots + p_q - q.\]

For one of the terms (\ref{terms}) to be a nonzero scalar multiple of $u_{-1} \mathbf{1} = u$, which is homogeneous of weight $\mathrm{wt}\, u$, we would have to have some linear component satisfying 
\begin{eqnarray*}
 \mathrm{wt} \, u &=& \mathrm{wt} \, {}_{\circ}^{\circ} \, u_{j_1}\cdots u_{j_m}\, {}_{\circ}^{\circ} \, u_{-p_1} \cdots u_{-p_q} \mathbf{1} \\
 &=& (m + q) \mathrm{wt} \, u  - j_1 - \cdots - j_m + p_1 + \cdots + p_q - m - q,
\end{eqnarray*}
that is
\begin{equation}\label{weights}
 j_1+\cdots+j_m =  (m + q -1) \mathrm{wt} \, u  + p_1 + \cdots + p_q - m - q.
\end{equation}
This with Eqn. (\ref{equation-for-sum2}) would imply 
\begin{equation}\label{weights2}
-2n-1-s  \leq  (m + q -1) \mathrm{wt} \, u  + p_1 + \cdots + p_q - m - q \leq (\mathrm{wt} \, u - 1)m -n - 1
\end{equation}
or equivalently
\begin{equation}\label{weights3}
-2n-s - (\mathrm{wt} \, u - 1) m   \leq (q -1) \, \mathrm{wt} \, u   + p_1 + \cdots + p_q  - q  + 1 \leq  -n <0 .
\end{equation}

If $q>0$, then since $p_1 + \cdots + p_q  - q \geq 0$, and $(q-1)  \, \mathrm{wt} \, u \geq 0$ when $q>0$, we have that 
\[ (q -1)  \mathrm{wt} \, u   + p_1 + \cdots + p_q  - q  + 1 \geq   1\]
which can not be less than or equal to $-n$ in Eqn.\ (\ref{weights3}).  Thus there exist no such combination of $p_1, \dots, p_q$ with $q>0$ such that $u_{-1} \mathbf{1}$ appears as a linear component of $v \circ_n w$.

If $q = 0$, i.e. if $w = \mathbf{1}$, then Eqn.\ (\ref{weights3}) becomes 
\begin{equation}\label{weights4}
-2n-s - (\mathrm{wt} \, u - 1) m   \leq   -  \mathrm{wt} \, u   + 1 \leq  -n <0,
\end{equation}
which is possible if and only if $\mathrm{wt} \, u \geq n + 1$, proving  that $(\mathrm{wt} \, u) u_{-1}{\bf{1}}+ u_{-2} {\bf{1}}\notin O_n^\circ(V)$ if $1 \leq \mathrm{wt} \, u \leq n$.

If $\mathrm{wt} \, u = 0$, then as previously noted $(L(0)+L(-1)) u_{-1}{\bf{1}}= (\mathrm{wt} \, u) u_{-1}{\bf{1}}+ u_{-2}{\bf{1}} = u_{-2} \mathbf{1}$, and thus in this case we have to show that $u_{-2} \mathbf{1} \notin O_n^\circ(V)$, which in particular, imposes the condition that $u_{-2} \mathbf{1} \neq 0$.    In this case, for one of the terms  (\ref{terms})  to be a nonzero scalar multiple of $u_{-2}{\bf{1}}$, we would have to have some $u_{-p_1} \cdots u_{-p_q} \mathbf{1}$  such that 
\begin{eqnarray*}
 \mathrm{wt} \, u_{-2} \mathbf{1}  &=& \mathrm{wt} \, {}_{\circ}^{\circ} \, u_{j_1}\cdots u_{j_m}\, {}_{\circ}^{\circ} \, w \\
 &=& (m + q) \mathrm{wt} \, u  - j_1 - \cdots - j_m + p_1 + \cdots + p_q - m - q,
\end{eqnarray*}
that is
\begin{equation}\label{weights-0}
 j_1+\cdots+j_m =  p_1 + \cdots + p_q - m - q -1 .
\end{equation}
This with Eqn. (\ref{equation-for-sum2}) would imply 
\begin{equation}\label{weights2-0}
-2n-1-s  \leq   p_1 + \cdots + p_q - m - q - 1  \leq -m -n - 1
\end{equation}
or equivalently
\begin{equation}\label{weights3-0}
-2n-s + m   \leq   p_1 + \cdots + p_q  - q   \leq  -n <0
\end{equation}
However, since $p_1 + \cdots + p_q  - q \geq 0$, this is impossible,   proving that  if $\mathrm{wt} \, u = 0$,  no $v,w \in V$ exist such that  $v\circ_n w$ has a term that is a nonzero scalar multiple of $u = u_{-2}{\bf 1}$. This proves that $(\mathrm{wt} \, u) u_{-1}{\bf{1}}+ u_{-2} {\bf{1}}\notin O_n^\circ(V)$ if $\mathrm{wt} \, u = 0$, and $u_{-2} \mathbf{1} \neq 0$.
\end{proof}

\begin{rem}\label{OL-remark}
{\em As stated, Theorem \ref{clarify-definition-thm} applies directly to vertex operator algebras strongly generated by a vector of weight greater than or equal to one, e.g.  the  rank one Heisenberg vertex operator algebra, or the Virasoro vertex operator algebra, or indirectly for such  positive weight elements $u \in V$ where there is not linear dependence between $(L(-1) + L(0))u$ and circle products of vectors in $V$ not generated by $u$,  in particular when the level $n$ of the Zhu algebra in consideration becomes greater than or equal to the weight of the vector being considered.  The Theorem also pertains to the case when $V$ has a nontrivial vacuum (weight zero) space containing any vectors $u$ such that  $u_{-2} {\bf 1} \neq 0$.   Thus there are common settings in which a vertex operator algebra $V$ has the property that $O^L (V) \not\subset O_n^\circ(V)$.   In practice, the theorem implies that when determining $A_n(V)$ for $n \in \mathbb{Z}_+$ and for $V$ any vertex operator algebra, one should at least single out vectors of positive weight less than the level $n$ being considered, any nontrivial weight zero vectors, as well as potentially other vectors, and add in these $(L(-1) + L(0))u \approx 0$  relations that may not arise from $O_n^\circ(V)$ but nevertheless arise from $O^L(V) \subset O_n(V)$.     That is, one should expect to use the fact that these can give nontrivial extra relations in $A_n(V)$ when determining a minimal set of generators and relations of $A_n(V)$.  For instance, the presence of these $(L(-1) + L(0)) v \in 
O_n(V)$ is used heavily in the constructions of the higher level Zhu algebras in \cite{BVY-Heisenberg}, \cite{BVY-Virasoro},  and \cite{AB-details-of-Heisenberg}. }
\end{rem}

\section{Using $A_{n-1}(V)$ to calculate $A_n(V)$}\label{technique-section} 

In this section, we provide some further techniques which have proved helpful in calculating higher level Zhu algebras once a reasonable set of generators has been determined by using the techniques of the previous sections.  The main two techniques in calculating $A_n(V)$, assuming $A_{n-1}(V)$ is already known,  are:  I. Make use of the relationship between these two algebras which imposes restrictions on the possible relations between the generators of $A_n(V)$ in terms of the relations of these generators in $A_{n-1}(V)$; II. Make use of the action of the zero modes of the generators of $A_n(V)$ on the $V$-modules already determined by $A_{n-1}(V)$ to help further determine relations amongst the generators in $A_n(V)$.  

We present two ``Theorems" below, which gather together the pertinent facts or general principles that allow Techniques I and II to be of use.  For Technique I, we note the following:

\begin{thm}\label{use-lower-A-thm}
Suppose that $S_n = \{v_j + O_n(V)\}_{j \in J}$ for $v_j \in V$ is a set of generators for $A_{n}(V)$ and let $R= \mathbb{C}\langle x_j \; | \; j \in J \rangle$ denote the free algebra generated by the formal variables $x_j$ for $j \in J$.  Then 

(i)  $S_{n-1}$ is a generating set for $A_{n-1}(V)$.

(ii) There exists a two-sided ideal $I_{n-1}$ of $R$ such that $A_{n-1}(V) \cong R/I_{n-1}$.

(iii) There exists a two-sided ideal $I_n$ of $R$ such that $A_n(V) \cong R/I_n$.

(iv) $I_n \subset I_{n-1}$.  
\end{thm}

\begin{proof} From Proposition 2.4 in \cite{DLM}, we have the surjective homomorphism given by Eqn.\ (\ref{surjection}), which follows from the fact that $O_n(V) \subset O_{n-1}(V)$.  If $S_n$ generates $A_n(V)$, then there is a surjective homomorphism from $R$ onto $A_n(V)$ with $x_j \mapsto v_j + O_n(V)$ whose kernel is a two-sided ideal $I_n$, giving $A_n(V) \cong R/I_n$.  Then the composition of the surjection of $R$ onto $A_n(V)$ and the surjection of $A_n(V)$ onto $A_{n-1}(V)$ gives a surjective homomorphism from $R$ onto $A_{n-1}(V)$ with $x_j \mapsto v_j + O_{n-1}(V)$ whose kernel is $I_{n-1}$, so that $A_{n-1}(V) \cong R/I_{n-1}$, and such that $I_n \subset I_{n-1}$.  
\end{proof}

For Technique II, we note the following Theorem

\begin{thm}\label{use-zero-mode-thm} 
Suppose that $S_n = \{v_j +O_n(V)\}_{j \in J} \subset A_n(V)$ is a set of generators for $A_{n}(V)$ and let $R= \mathbb{C}\langle x_j \; | \; j \in J \rangle$ denote the free algebra generated by the formal variables $x_j$ for $j \in J$.   If $W = L_{n-1}(U)$ is a $V$-module induced at level $n-1$ from the $A_{n-1}(V)$-module $U$, $P(x_j \; | \;  j \in J)$ is a word in $R$, and $P(v_j + O_n(V) \; | \; j \in J) \in I_n$ where $A_n \cong R/I_n$, then $o(P).W(k) = 0$ for $k = 0, \dots, n$.  
\end{thm}

\begin{proof}  Recall from Section 2, or \cite{DLM}, that for any $\mathbb{N}$-gradable $V$-module $W$, the subspace $\Omega_n(W)$ which contains $\coprod_{k = 0}^n W(k)$ (see \cite{BVY}), is an $A_n(V)$-module with action given via zero modes, i.e. $(v + O_n(V)).w = o(v)w = v_{\mathrm{wt} \, v - 1}w$, where $v_{\mathrm{wt}, v - 1}$ acting on $W$ is defined via the $V$-module structure of $W$.  

Thus if $P \in I_n$, then $o(P)$ acts trivially on any $A_n(V)$-module, in particular on any $\coprod_{k = 0}^n W(k)$ for $W$ a $V$-module induced from $A_{n-1}(V)$.  
\end{proof} 

In particular, if a relation $P$ can be found or guessed and proved to hold up to lower order terms in the total degree of the words in the generators $S$, then often acting via the zero modes of the generators, one can establish what the lower degree terms of $P$ are exactly.  In addition, using the fact that $P \in I_n \subset I_{n-1}$ as established in Theorem \ref{use-lower-A-thm}, allows one to further limit the possible form of these lower degree terms, thus helping establish the relation.  

We illustrate these techniques in detail in \cite{AB-details-of-Heisenberg} and in the following section by giving some of the results which were used in \cite{AB-details-of-Heisenberg} to calculate the level two Zhu algebra for the Heisenberg vertex operator algebra using the relations of the level two generators inside the level one Zhu algebra for the Heisenberg vertex operator algebra which was calculated in \cite{BVY-Heisenberg}.

\section{Applications}

The applications of the results of the previous sections are wide ranging.  In particular they apply to the subalgebra of $A_n(V)$ generated by the conformal element $\omega$ of $V$, for any vertex operator algebra $V$, since $\langle \omega \rangle$ is a strongly generated vertex operator algebra.  They also apply to any elements $u \in V$ that generate a Heisenberg vertex subalgebra. We discuss these two settings below. 

\subsection{The Heisenberg vertex operator algebra}\label{Heisenberg-section}

We denote by $\mathfrak{h}$ a one-dimensional abelian Lie algebra spanned by $\alpha$ with a bilinear form $\langle \cdot, \cdot \rangle$ such that $\langle \alpha, \alpha \rangle = 1$, and by
\[
\hat{\mathfrak{h}} = \mathfrak{h}\otimes \C[t, t^{-1}] \oplus \C \mathbf{k}
\]
the affinization of $\mathfrak{h}$ with bracket relations
\[
[a(m), b(n)] = m\langle a, b\rangle\delta_{m+n, 0}\mathbf{k}, \;\;\; a, b \in \mathfrak{h},
\]
\[
[\mathbf{k}, a(m)] = 0,
\]
where we define $a(m) = a \otimes t^m$ for $m \in \mathbb{Z}$ and $a \in \mathfrak{h}$.

Set
\[
\hat{\mathfrak{h}}^{+} =\mathfrak{h} \otimes   t\C[t] \qquad \mbox{and} \qquad \hat{\mathfrak{h}}^{-} = \mathfrak{h} \otimes t^{-1}\C[t^{-1}].
\]
Then $\hat{\mathfrak{h}}^{+}$ and $\hat{\mathfrak{h}}^{-}$ are abelian subalgebras of $\hat{\mathfrak{h}}$.  Consider the induced $\hat{\mathfrak{h}}$-module given by 
\[
M(1) = U(\hat{\mathfrak{h}})\otimes_{U(\C[t]\otimes \mathfrak{h} \oplus \C c)} \C{\bf 1} \simeq S(\hat{\mathfrak{h}}^{-}) \qquad \mbox{(linearly)},
\]
where $U(\cdot)$ and $S(\cdot)$ denote the universal enveloping algebra and symmetric algebra, respectively, $\mathfrak{h} \otimes \C[t]$ acts trivially on $\mathbb{C}\mathbf{1}$ and $\mathbf{k}$ acts as multiplication by $1$.   Then $M(1)$ is a vertex operator algebra, often called the {\it vertex operator algebra associated to the rank one Heisenberg}, or the {\it rank one Heisenberg vertex operator algebra}, or the {\it one free boson vertex operator algebra}.  Here, the Heisenberg Lie algebra in question being precisely $\hat{\mathfrak{h}} \diagdown \mathbb{C}\alpha(0)$.

Any element of $M(1)$ can be expressed as a linear combination of elements of the form
\begin{equation}\label{generators-for-V}
\alpha(-k_1)\cdots \alpha(-k_j){\bf 1}, \quad \mbox{with} \quad  k_1 \geq \cdots \geq k_j \geq 1.
\end{equation}
In other words, $M(1)$ is strongly generated by $\alpha(-1) \mathbf{1}$ implying that most of the results of this paper pertain to $M(1)$.  

The conformal element for $M(1)$ is given by $\omega = \frac{1}{2} \alpha(-1)^2{\bf 1}$.
However the element $\omega_a = \frac{1}{2} \alpha(-1)^2{\bf 1} + a \alpha(-2) \mathbf{1}$
for any $a \in \mathbb{C}$, also gives $M(1)$ with the conformal element $\omega_a$ the structure of a vertex operator algebra with central charge $c = 1 - 12 a^2$.  We denote this vertex operator algebra with shifted conformal elements by $M_a(1)$.    So, for instance $M(1) = M_0(1)$.

Next, we recall the construction of $A_0(M_a(1))$ from  \cite{FZ} and the construction of $A_1(M_a(1))$ from \cite{BVY}:

\begin{thm}\label{A_0-theorem} \cite{FZ} As algebras 
\[A_0(M_a(1)) \cong \mathbb{C}[x,y]/(x^2-y) \cong \mathbb{C}[x]\] 
under the identifications 
\[\alpha(-1)\mathbf{1} + O_0(M_a(1)) \longleftrightarrow x, \qquad \mbox{and} \qquad  \alpha(-1)^2\mathbf{1} + O_0(M_a(1)) \longleftrightarrow y.\]  
In addition, there is a bijection between isomorphism classes of irreducible admissible $M_a(1)$-modules and irreducible $\mathbb{C}[x]$-modules given by $M_a(1,\lambda) \longleftrightarrow \mathbb{C}[x]/(x - \lambda)$. 
\end{thm}

\begin{thm}\label{A_1-theorem}\cite{BVY-Heisenberg}
Let $I$ be the ideal generated by the polynomial $(x^2-y)(x^2-y+2)$ in $\mathbb{C}[x, y]$.  
Then we have the following isomorphisms of algebras 
\begin{eqnarray}\label{first-characterization-level-one}
A_1(M_a(1)) &\cong& \mathbb{C}[x, y]/I \ \cong \ \mathbb{C}[x,y]/(x^2 - y) \oplus \mathbb{C}[x,y]/(x^2 - y + 2) \\
& \cong & A_0(M_a(1)) \oplus \mathbb{C}[x,y]/(x^2 - y+ 2)  \ \cong \  \mathbb{C}[x] \oplus \mathbb{C} [x] \label{second-characterization}
\end{eqnarray}
under the identifications
\[
\alpha(-1){\bf 1} + O_1(M_a(1)) \longleftrightarrow x + I, \quad \qquad  \alpha(-1)^2{\bf 1} + O_1(M_a(1)) \longleftrightarrow y + I.
\]
\end{thm}

More recently, using results from this paper, we prove in \cite{AB-details-of-Heisenberg} the following:

\begin{thm}\label{main-generators-thm}\cite{AB-details-of-Heisenberg}
The algebra $A_2(M_a(1))$ is generated by $v + O_2(M_a(1))$ for $v$ in  
\[S = \{ \alpha(-1) {\bf 1}, \, \alpha(-1)^2 {\bf 1}, \,  \alpha(-1) \alpha(-4) {\bf 1}, \, \alpha(-1)^2\alpha(-4){\bf 1}, \,  \alpha(-1)\alpha(-4)^2{\bf 1} \}.\]
\end{thm}

We introduce the notation:
\begin{eqnarray}
x_n &=& \alpha(-1) \mathbf{1} + O_n(M_a(1)) \\
y_n  &=& \alpha(-1)^2 \mathbf{1} + O_n(M_a(1)) \\
\tilde y_n &=& \alpha(-1)\alpha(-4) \mathbf{1} + O_n(M_a(1)) \\
z_n  &=& \alpha(-1)^2 \alpha(-4) \mathbf{1} + O_n(M_a(1)) \\
\tilde z_n &=& \alpha(-1)\alpha(-4)^2  \mathbf{1} + O_n(M_a(1)).
\end{eqnarray}

Using Technique I of Section \ref{technique-section}, Theorem \ref{use-lower-A-thm}, and results of the proof of Theorem \ref{A_1-theorem} in \cite{BVY-Heisenberg},   in \cite{AB-details-of-Heisenberg}  we prove

\begin{prop}\label{use-level-one-prop}\cite{AB-details-of-Heisenberg}
For $V = M_a(1)$, and $R_1 = \mathbb{C}[x_1,y_1]\langle \tilde y_1, z_1, \tilde z_1\rangle$, 
we have
\[A_1(V) \cong R_1/I_1 \]
where $I_1$ is the two-sided ideal 
\begin{equation}\label{I_1}
I_1 = \big( (x_1^2  -y_1)(x_1^2 - y_1  + 2) , (x_1^2 -2y_1 -  \tilde{y}_1),  (4x_1^3 -5x_1y_1 - z_1), \\
 (3x_1^3 - 4x_1y_1 + \tilde{z}_1) \big).
\end{equation} 
\end{prop} 

Motivated by this Proposition \ref{use-level-one-prop}, and the action of the zero modes for the generating set $S$ for $A_2(M_a(1))$ on the irreducible modules already determined by $A_1(M_a(1))$ as proscribed by Technique II and Theorem \ref{use-zero-mode-thm}, we make the change of variables:
\begin{eqnarray}\label{change-var1-n}
Y_n &=& \frac{1}{12} (x_n^2 - 2y_n - \tilde y_n)\\
Z_n &=&   \frac{1}{32} ( x_n^3 + 2x_n\tilde y_n + \tilde z_n) \\
&=&  \frac{1}{32} \left( (3x_n^3 - 4x_ny_n + \tilde z_n) - 2x_n(x_n^2 - 2y_n - \tilde y_n) \right), \\
\ \ \  W_n &=&  -\frac{1}{40}  \left(2z_n + \tilde z_n + 2x_ny_n - 2x_n \tilde y_n - 3x_n^3 \right)  \label{change-var3-n} \\
 &=&  -  \frac{1}{40} \left( 2x_n(x_n^2 - 2y_n - \tilde y_n)-  2(4x_n^3 - 5x_ny_n - z_n) + (3x_n^3 - 4x_n y_n + \tilde z_n)\right) \nonumber,
\end{eqnarray}
so that $R_n = \mathbb{C}[x_n,y_n]\langle \tilde y_n, z_n, \tilde z_n \rangle = \mathbb{C}[x_n,y_n]\langle Y_n, Z_n, W_n \rangle$.

With these changes of variables, in \cite{AB-details-of-Heisenberg} we note the following Corollary to Proposition \ref{use-level-one-prop} 

\begin{cor}\label{level-one-cor}\cite{AB-details-of-Heisenberg}
For $V = M_a(1)$ and $R  = \mathbb{C}[x,y]\langle Y, Z, W \rangle$ we have $A_1(V) \cong R/I_1$ where $I_1$ is the two-sided ideal 
\begin{equation}
I_1 = \big( (x_1^2 - y_1)(x_1^2 - y_1 + 2), Y_1, Z_1, W_1 \big).
\end{equation}
\end{cor}

Then in \cite{AB-details-of-Heisenberg}, motivated by the zero-mode action of these generators on the already determined irreducible $M_a(1)$-modules, we prove relations for the generators of $A_2(M_a(1))$ up to lower order terms in $F_r(\mathbf{1})$ for $r$ less than the top degree of the conjectured relations.  In particular we prove 

\begin{lem}\label{lower-order-lemY} \cite{AB-details-of-Heisenberg} 
\begin{eqnarray}
(x_2^2 - y_2) Y_2, \ Y_2^2  & \in & O_2(V) + F_3(\mathbf{1}) , \\
(x_2^2 - y_2)Z_2, \ (x_2^2 - y_2)W_2, \ Z_2Y_2, \ Y_2W_2  & \in & O_2(V) + F_4(\mathbf{1}), \\
(x_2^2 - y_2)^3, \ Z_2^2, \ W_2^2 &\in & O_2(V) + F_5(\mathbf{1}).
\end{eqnarray}
\end{lem}

Then using the zero-mode action of these generators on the irreducible $M_a(1)$-modules, we prove 

\begin{thm}\label{A_2-theorem}\cite{AB-details-of-Heisenberg}
  Let   $\mathbb{C}[x, y]\langle Y, Z, W \rangle$ denote the algebra generated over $\mathbb{C}$ by the two commuting variables $x$ and $y$, and three non-commuting variables $Y, Z$, and $W$.  Let $I_2$ be the ideal generated by the polynomials
\begin{eqnarray}
\ \ \ \ & (x^2-y)(x^2-y+2)(x^2-y+4), \ \ (x^2 - y + 4)Y, \ \ (x^2 - y + 4)Z, \ \ (x^2 - y + 4)W, &\\
  &Y^2 - Y,  \ \ Z^2, \ \ W^2, \ \ ZY, \ \ YW, \  \ ZW - Y, & \\
 &  YZ - Z, \ \ WY - W, \ \  Y + WZ - \frac{1}{8}(x^2 - y)(x^2 - y + 2)    .&
   \end{eqnarray}

Then we have the following isomorphism of algebras  
\begin{equation}
A_2(M_a(1)) \cong \mathbb{C}[x, y]\langle Y, Z, W \rangle/I_2 \\
\end{equation}
under the identification 
\begin{equation}\label{identification1}
\alpha(-1){\bf 1} + O_2(M_a(1)) \longleftrightarrow x + I_2, \quad \alpha(-1)^2{\bf 1} + O_2(M_a(1)) \longleftrightarrow y + I_2,
\end{equation}
\begin{equation}
\alpha(-1)\alpha(-4){\bf 1} + O_2(M_a(1)) \longleftrightarrow \tilde y + I_2, \quad \alpha(-1)^2\alpha(-4){\bf 1} + O_2(M_a(1)) \longleftrightarrow z + I_2,
\end{equation}
\begin{equation}
\alpha(-1)\alpha(-4)^2{\bf 1} + O_2(M_a(1)) \longleftrightarrow \tilde z + I_2. \label{identification3}
\end{equation}
and the change of variables 
\begin{equation}
\ Y = \frac{1}{12} (x^2 - 2y - \tilde y),  \ \ Z =\frac{1}{32} ( x^3 + 2x\tilde y + \tilde z), \ \ W =  -\frac{1}{40}  \left(2z + \tilde z + 2xy -2x \tilde y - 3x^3 \right).
\end{equation}

Furthermore
\begin{eqnarray}
A_2(M_a(1)) & \cong& \mathbb{C}[x]\oplus \mathbb{C}[x]\oplus (\mathbb{C}[x]\otimes M_2(\mathbb{C}))  \label{second-characterization-level-two}  \\
& \cong & A_1(M_a(1))\oplus (\mathbb{C}[x]\otimes M_2(\mathbb{C})),
\label{third-characterization-level-two} 
\end{eqnarray} 
where $M_2(\mathbb{C})$ denotes the algebra of $2\times 2$ complex matrices. 
\end{thm}

\subsection{Generators and relations for the algebra $A_n(M_a(1))$}\label{A_n-construction-section}

As an illustration of applications of this paper, we note that $M_a(1)=\langle \alpha(-1){\bf 1}\rangle$ and satisfies the permutation property. Therefore,  the results of Sections \ref{general-n-section}  and 4 can be used in the calculation of the higher level Zhu algebras  $A_n(M_a(1))$, for $n>0$, and as mentioned above, have been used by the authors to explicitly determine the structure of $A_2(M_a(1))$ \cite{AB-details-of-Heisenberg},   using only the internal structure of $M_a(1)$ to obtain a minimal set of generators, and then some basic facts about irreducible modules for $M_a(1)$ to prove the relations.  

In particular, the recursion given in Proposition \ref{recursion-n} can be used to reduce the collection of potential generators in order to determine a finite generating set for $A_n(M_a(1))$. Lemma \ref{Y+-lemma} allows one to obtain expressions that are equivalent modulo $O_n^\circ(V)$ to $Y^+(\alpha(-m){\bf 1}, x)$ and are rational functions of $x$, which is helpful in simplifying the calculations of circle products, $u\circ_nv$, that can be used in determining the structure of $A_n(M_a(1))$. Proposition \ref{general-generators-prop} and Theorem \ref{reducing-general-generators-thm} further reduce the set of generators. Eqns.\ \eqref{reduce-equation} and \eqref{reduce-equation-n=1} are important in simplifying many of the expressions needed to determine relations satisfied by elements of $A_2(M_a(1))$, and one expects to similarly use such expressions to calculate $A_n(M_a(1))$ for $n>2$. 

The multiplication formulas from Section \ref{mult-section} and the subsequent results about generators are also heavily used in obtaining the structure of $A_2(M_a(1))$ and can be used to calculate $A_n(M_a(1))$ for $n>2$. For instance, Theorem \ref{last-generator-theorem} applied to the case when $V=M_a(1)$  gives the following result:

\begin{thm}\label{last-generator-theorem-heisenberg} For $n \in \mathbb{Z}_+$,
 $A_n(M_a(1))$ is generated by elements of the form 
\[ \alpha({-2n})^{i_{2n}}\alpha({-2 n + 1})^{i_{2n -1}} \cdots \alpha({ - 1})^{i_1}\mathbf{1}+ O_n(V) ,  \mbox{and} \ \alpha({-1})^{j} \mathbf{1} + O_n(V),\]
for $i_1, i_2, \dots, i_{2n} \in \mathbb{N}$, with $0 \le i_1 \leq n$, and $j \in \mathbb{Z}_+$.
\end{thm}} 
This immediately reduces the collection of elements one must consider in determining finite generating sets for the higher level Zhu algebras for $M_a(1)$. 

Theorem \ref{clarify-definition-thm} shows that vectors of the form $(L(-1)+L(0))v$, $v\in M_a(1),$ must be included in the definition of $O_n(M_a(1))$ when one defines $A_n(M_a(1))$ for any $n>0$ and, as is illustrated in \cite{AB-details-of-Heisenberg}, provides important insight into determining the structure of these higher level Zhu algebras. 

In \cite{AB-details-of-Heisenberg}, we make the following conjecture for the general structure of the higher level Zhu algebras $A_n(M_a((1))$, which has been proven for $n = 1$ in  \cite{BVY-Heisenberg} and for $n= 2$ in \cite{AB-details-of-Heisenberg}

\begin{conj} For $n \in \mathbb{Z}_+$
\begin{eqnarray*}
A_n(M_a(1)) &\cong& A_{n-1}(M_a(1)) \oplus \left( \mathbb{C}[x] \otimes M_{p(n)}(\mathbb{C}) \right)\\
&\cong& \mathbb{C}[x] \oplus \mathbb{C}[x] \oplus \left( \mathbb{C}[x] \otimes M_2(\mathbb{C}) \right) \oplus \left(\mathbb{C}[x] \otimes M_3(\mathbb{C}) \right) \oplus  \left(\mathbb{C}[x] \otimes M_5(\mathbb{C}) \right) \\
& & \quad \oplus \cdots  \oplus  \left(\mathbb{C}[x] \otimes M_{p(n)} (\mathbb{C}) \right) 
\end{eqnarray*} 
where $p(n)$ denotes the unordered partitions of $n$ into nonnegative integers, and $M_{p(n)}(\mathbb{C})$ denotes the algebra of $p(n) \times p(n)$ matrices.  
\end{conj}

\subsection{The Virasoro vertex operator algebra}
The Virasoro vertex operator algebra is another example of a vertex operator algebra for which the results of this paper can be readily and easily applied to  aid in determining the structure of higher level Zhu algebras.

Let $\mathcal{L}$ be the Virasoro algebra with central charge $\mathbf{c}$, that is, $\mathcal{L}$ is the vector space with basis $\{\bar{L}_n \,|\, n\in\Z\}\cup \{\mathbf{c}\}$ with bracket relations 
\begin{align*}
[\bar{L}_m,\bar{L}_n]=(m-n)\bar{L}_{m+n}+\frac{m^3-m}{12} \delta_{m+n,0} \, \textbf{c},\quad\quad  [\textbf{c},\bar{L}_m]=0  
\end{align*}  
for $m,n\in \Z$.  Here we use a bar over the Virasoro generators to distinguish between these Virasoro elements and the functor $L_n$ defined earlier.

Let $\mathcal{L}^{\geq 0}$ be the Lie subalgebra with basis $\{ \bar{L}_n  \,|\, n\geq 0 \} \cup \{\mathbf{c}\}$, and let $\mathbb{C}_{c,h}$ be the $1$-dimensional $\mathcal{L}^{\geq 0}$-module where $\mathbf{c}$ acts as $c$ for some $c\in \mathbb{C}$, $\bar{L}_0$ acts as $h$ for some $h\in \mathbb{C}$, and $\bar{L}_n$ acts trivially for $n\geq 1$. Form the induced $\mathcal{L}$-module
\begin{equation}\label{V-module-M}
M(c,h)= \mathcal{U}(\mathcal{L})\otimes_{\mathcal{L}^{\geq 0}} \mathbb{C}_{c,h} .
\end{equation}
We shall write $L(n)$ for the operator on a Virasoro module corresponding to $\bar{L}_n$, and $\mathbf{1}_{c,h} = 1 \in \mathbb{C}_{c,h}$.  Then 
\[ V_{Vir}(c,0)= M(c,0)/\langle L(-1)\mathbf{1}_{c,0}\rangle
\] 
has a natural vertex operator algebra structure with vacuum vector $1=\mathbf{1}_{c,0}$, and conformal element $\omega=L(-2)\mathbf{1}_{c,0}$, satisfying  $Y(\omega,x)= \sum_{n\in\Z } L(n)x^{-n-2}$.  In addition, for each $h\in\C$, we have that $M(c,h)$ is an ordinary $V_{Vir}(c,0)$-module with $\mathbb{N}$-gradation 
\[M(c,h)=\coprod_{k\in \N} M(c,h)_k\] 
where $M(c,h)_k$ is the $L(0)$-eigenspace with eigenvalue $h + k$.  We say that  $M(c,h)_k$ has degree $k$ and weight $h+k$.   

We now fix $c \in \mathbb{C}$, and denote by $V$, the vertex operator algebra $V_{Vir}(c,0)$.  

It was shown in \cite{W} that
\begin{equation}
A_0(V) \cong \mathbb{C}[x,y]/(y-x^2 - 2x) \cong \mathbb{C}[x]
\end{equation}
under the identification 
\begin{eqnarray}
L(-2)\mathbf{1} + O_0(V) &\longleftrightarrow& x + (q_0(x,y)), \\ 
L(-2)^2\mathbf{1} + O_0(V) &\longleftrightarrow& y + (q_0(x,y)),
\end{eqnarray}  
where $q_0(x,y) = y - x^2 - 2x$.

It was shown in \cite{BVY-Virasoro} that 
\begin{equation}
A_1(V) \cong \mathbb{C}[x,y]/(y-x^2 - 2x)(y - x^2 -6x + 4)\cong \mathbb{C}[\tilde{x},\tilde{y}]/(\tilde{x}\tilde{y})
\end{equation}
under the identification 
\begin{eqnarray}
L(-2)\mathbf{1} + O_1(V) &\longleftrightarrow& x + (q_0(x,y)q_1(x,y) ), \\ 
L(-2)^2\mathbf{1} + O_1(V) &\longleftrightarrow& y + (q_0(x,y)q_1(x,y)),
\end{eqnarray}  
where  $\tilde{x} = q_0(x,y) = y - x^2 - 2x$, and $\tilde{y} = q_1(x,y) = y - x^2 - 6x + 4$.

Fix $V = V_{Vir}(c,0)$, and recall (e.g., \cite{LL}, \cite{KarL}) that $V$ is spanned by vectors of the form
\begin{equation}\label{span}
L(-k_1) \cdots L(-k_m) \mathbf{1} \qquad \mbox{for $k_1 \geq \cdots \geq k_m \geq 2$ and $m \in \mathbb{N}$} .
\end{equation}

For $v \in V$ and $r \in \N$, we define $F_r(v)$ to be the subspace of $V$ linearly spanned by all elements of the form
\[
L(-k_1)\cdots L(-k_m)v,
\]
where $k_1, \dots, k_m \geq 2$ and $m \leq r$. Thus $F_r(v) \subseteq F_{r+1}(v)$.  Then from \cite{BVY-Virasoro}, we have the following:

\begin{lem}\label{permute}
For $m \in \mathbb{Z}_+$, let $\sigma$ be a permutation on $m$ letters, and let $v \in V$. Then for $k_1, \dots, k_m \geq 2$,
\[
L(-k_1)\cdots L(-k_m)v - L(-k_{\sigma(1)})\cdots L(-k_{\sigma(m)})v \in F_{m-1}(v).
\]
In particular, 
\[  L(-k_1) \cdots L(-k_m)  v \in F_{m-1}(v) \]
if $k_j \leq 1$ for any $j = 1, \dots, m$.
\end{lem}

In particular, $V_{Vir}(c,0)$ satisfies the permutation property and is strongly generated by $\omega$. Therefore, the results listed in Section \ref{A_n-construction-section}, which we noted were important in determining the structure of $A_2(M(1))$ can also be applied to $A_n(V_{Vir}(c,0))$ for $n>1$.  For instance, Theorem \ref{last-generator-theorem} applied to the case $V_{Vir}(c,0)$ gives the following result:

\begin{thm}\label{last-generator-theorem-virasoro}
Let $V=V_{Vir}(c,0)$ and fix $n \in \mathbb{Z}_+$.  Then $A_n(V)$ is generated by elements of the form 
\[ L({-2n-1})^{i_{2n+1}} L({-2 n})^{i_{2n}} \cdots L({ - 2})^{i_2}\mathbf{1}+ O_n(V) ,  \mbox{and} \ L({-2})^{j} \mathbf{1} + O_n(V),\]
for $i_2, i_3, \dots, i_{2n+1} \in \mathbb{N}$, with $0 \leq 2i_2  \leq n$, and $j \in \mathbb{Z}_+$.  
\end{thm}
As in the case of $M_a(1)$, one can further reduce generators and calculate relations satisfied by these generators by using the multiplication formulas found in Section \ref{mult-section} as well as by applying Lemma \ref{Y+-lemma} to calculate formulas equivalent modulo $O_n^\circ(V)$ to various circle products, $u\circ_nv$. Finally, we remark that Theorem \ref{clarify-definition-thm} shows the importance of including the vectors of the form $(L(-1)+L(0))v$, $v\in V_{Vir}(c,0)$ in $O_n(V)$ for $n>1$.

\appendix
\section{A proof of the combinatorial identity Eqn.\  \eqref{combinatorial-formula} in Proposition \ref{recursion-n}}
\label{appendix}

To prove Eqn.\  \eqref{combinatorial-formula} in Proposition \ref{recursion-n}, we first prove the following lemma which involves a special case of a combinatorial identity that can be found in \cite{K}, Section 1.2.6,  formula I (24).  In \cite{K} there is an outline of the proof of this identity given in the exercises, however for completeness, we provide a proof here.

\begin{lem}\label{appendix-lemma}
For $r, m \in \mathbb{N}$ with $r-m \geq 0$ and $s\in \mathbb{Z}$,
\begin{equation}\label{combo-lemma-id1}
\sum_{l\in \mathbb{N}}(-1)^{l}\binom{s}{l}\binom{r-l}{m}=\binom{r-s}{r-m}.
\end{equation}
In particular this implies that for $n, k, j \in \mathbb{Z}$ such that $n \in \mathbb{N}$, $k >-n$, and $1 \leq j \leq n + k$,  setting $m=n+k-j$, $r=n+k-1$,  and $s=n+k$, we have
\begin{equation}\label{combo-lemma-id2}
\sum_{l\in \mathbb{N} }(-1)^{l}\binom{n+ k}{l}\binom{n + k - 1 -l}{n + k - j}=\binom{-1}{j-1} = (-1)^{j-1} .
\end{equation}
\end{lem}

\begin{proof} Note that by the definition of the binomial, we have
\begin{equation}\label{combo-id1}
\binom{p}{q} = (-1)^q \binom{q - p - 1}{q}.
\end{equation}

Also recall the Vandermonde identity (cf. \cite{K}), which can easily be derived by equating coefficients in the expansion of $(1 + x)^{i+ j} = (1+x)^i(1+x)^j$ for $i,j\in \mathbb{Z}$, is given by
\begin{equation}
\sum_{l \in \mathbb{N}} \binom{i}{l} \binom{j}{q-l} = \binom{i+j}{q}
\end{equation}
for $q \in \mathbb{N}$.  Setting $i = s$, $j = -m - 1$, for $s, m \in \mathbb{Z}$, and $q = r - m$ for $r\geq m$, we have
\begin{equation}\label{van2}
 \sum_{l\in \mathbb{N}} \binom{s}{l}  \binom{-m-1}{r-m -l}  =  \binom{s-m-1}{r-m}.
\end{equation}

Thus for $r, m \in \mathbb{N}$ with $r-m\geq 0$, and $s \in \mathbb{Z}$, using the fact that $\binom{p}{q} = \binom{p}{p-q}$ for $p = r-l$ and $q = r-l-m$, then using Eqn. (\ref{combo-id1}) in the case when $p = r - l$ and $q = r-l-m$, simplifying, then using Eqn. (\ref{van2}), and finally using Eqn. (\ref{combo-id1}) in the case when $p = r-s$ and $q = r-m$, we have 
\begin{eqnarray*}
\sum_{l\in \mathbb{N}} (-1)^{l}\binom{s}{l}\binom{r-l}{m} &=& \sum_{l \in \mathbb{N}}(-1)^{l}\binom{s}{l}\binom{r-l}{r-l-m}  \\
&=& \sum_{l \in \mathbb{N}} (-1)^{l}\binom{s}{l} (-1)^{r-l-m} \binom{-m-1}{r-l-m}\\
&=& (-1)^{r - m}  \sum_{l\in \mathbb{N}} \binom{s}{l}  \binom{-m-1}{r-l-m} \\
&=& (-1)^{r- m} \binom{s-m-1}{r-m} \\
&=& \binom{r-s}{r-m},
\end{eqnarray*}
proving (\ref{combo-lemma-id1}).  Then (\ref{combo-lemma-id2}) follows immediately using the substitution specified.
\end{proof}

\begin{prop}
\label{recursion-n-identity}
For $n \in \mathbb{N}$, $k \in \mathbb{Z}$ with $k > -n$, and $j = 1, \dots, n + k$ and $m \in \mathbb{N}$ with $m > 2n + k$, we have 
\begin{equation}\label{combinatorial-formula2}
 \sum_{i = 0}^{ n + k}  (-1)^{i+1} \binom{ n + k }{i} 
\binom{m-i-n-1}{j-1} \binom{m-i-n-j-1}{n+ k-j}  =0.
\end{equation}
\end{prop}

\begin{proof}
Since $m -i - n - 1 \geq 0$, the second binomial coefficient in Eqn.\ (\ref{combinatorial-formula2}) will be zero and thus will not contribute to the sum if $j-1 > m -i - n - 1$, i.e. if $i > m - j - n$.  More importantly, for certain $j = 1, \dots, n + k$, $m > 2n + k$, $k >-n$, and $i = 0, \cdots, n +k$, it can happen that $m - i - n - j = 0$, i.e. for certain quadruples $(j, k, m, n)$, we can have $m-n-j \in \{ 0, \dots, n+k\}$.  For fixed $j, k, m$, and $n$, we define 
\[
\epsilon (j,k,m,n) = \left \{ \begin{array}{ll} 
0 & \mbox{if $m-n-j \notin \{ 0, \dots, n+k\}$}\\
1 & \mbox{if $m-n-j \in \{0, \dots, n+k\}$}
\end{array} \right.  .\]
Thus for fixed $n \in \mathbb{N}$, and $j = 1, \dots, n + k$, $k >-n$, and $m > 2n + k$, we have
\begin{eqnarray*}
\lefteqn{ \sum_{i = 0}^{ n + k}  (-1)^{i+1} \binom{ n + k }{i} 
\binom{m-i-n-1}{j-1} \binom{m-i-n-j-1}{n+k-j}}\\
&=& \sum_{i=0}^{n+k} (-1)^{i+1}\binom{n+k}{i}\frac{(m-i -n-1)(m-i-n-2)\cdots (m-i-n-j + 1) }{(j-1)! (n + k-j)!}\\
&& \quad \cdot (m-i-n-j-1) \cdots (m-i-2n-k)\\
&=&  \sum_{\stackrel{0\leq i \leq n + k}{i \neq m - j - n}}  (-1)^{i+1}\binom{n+k}{i}\frac{(m-i -n-1)(m-i-n-2)\cdots (m-i-2n-k)}{(j-1)! (n + k-j)! (m-i-n-j) } \\
& & \quad + \, \epsilon(j,k,m,n)  (-1)^{m + k + 1} \binom{n+k}{m -n -j}\\
&=& \frac{1}{(j-1)! (n+k - j)!} \bigg(\sum_{\stackrel{0\leq i \leq n + k}{i \neq m - j - n}}  (-1)^{i+1}\binom{n+k}{i}\Big(\frac{d}{dx}\Big)^{n+k}\frac{x^{m-i-n-1}}{m-i-n-j}\bigg)\bigg\rvert_{x=1}\\
& & \quad + \, \epsilon(j,k,m,n)  (-1)^{m + k + 1} \binom{n+k}{m -n -j}\\
&=& \frac{1}{(j-1)! (n+k - j)!} \bigg( \! \Big(\frac{d}{dx}\Big)^{n+k} \biggl(\sum_{\stackrel{0\leq i \leq n + k}{i \neq m - j - n}}  (-1)^{i+1}\binom{n+k}{i} \frac{x^{m-i-n-1}}{m-i-n-j}\bigg) \! \biggr) \bigg\rvert_{x=1}\\
& & \quad + \, \epsilon(j,k,m,n)  (-1)^{m + k + 1} \binom{n+k}{m -n -j}\\
&=& \frac{1}{(j-1)! (n+k - j)!} \bigg( \!  \Big(\frac{d}{dx}\Big)^{n+k} \biggl(x^{j-1} \sum_{\stackrel{0\leq i \leq n + k}{i \neq m - j - n}}  (-1)^{i+1}\binom{n+k}{i} \frac{x^{m-i-n-j}}{m-i-n-j}\bigg)  \! \biggr) \bigg\rvert_{x=1}\\
& & \quad + \, \epsilon(j,k,m,n)  (-1)^{m + k + 1} \binom{n+k}{m -n -j}\\
&=& \frac{1}{(j-1)! (n+k - j)!} \bigg(\! \Big(\frac{d}{dx}\Big)^{n+k} \biggl(x^{j-1}  \sum_{\stackrel{0\leq i \leq n + k}{i \neq m - j - n}}  (-1)^{i+1}\binom{n+k}{i} \int  x^{m-i-n-j-1} dx \bigg) \! \biggr) \bigg\rvert_{x=1}\\
& & \quad + \,  \epsilon(j,k,m,n)  (-1)^{m + k + 1} \binom{n+k}{m -n -j}\\
&=& \frac{1}{(j-1)! (n+k - j)!} \bigg( \! \Big(\frac{d}{dx}\Big)^{n+k} \biggl(x^{j-1}  \int x^{m-2n-k-j-1}  \\
& & \quad \cdot \bigg( \sum_{\stackrel{0\leq i \leq n + k}{i \neq m - j - n}}  (-1)^{i+1}\binom{n+k}{i} x^{n+k-i}  \bigg) dx \bigg)\!  \biggr) \bigg\rvert_{x=1} +  \epsilon(j,k,m,n)  (-1)^{m + k + 1} \binom{n+k}{m -n -j}\\
&=& \frac{-1}{(j-1)! (n+k - j)!} \bigg( \!  \Big(\frac{d}{dx}\Big)^{n+k} \biggl(x^{j-1}  \int x^{m-2n-k-j-1}  \\
& & \quad \cdot \bigg(\sum_{\stackrel{0\leq i \leq n + k}{i \neq m - j - n}}  (-1)^{i}\binom{n+k}{i} x^{n+k-i}  \bigg) dx \bigg) \! \biggr) \bigg\rvert_{x=1}  +   \epsilon(j,k,m,n)  (-1)^{m + k + 1} \binom{n+k}{m -n -j}\\
&=& \frac{-1}{(j-1)! (n+k - j)!} \bigg( \! \Big(\frac{d}{dx}\Big)^{n+k} \biggl(x^{j-1}  \int x^{m-2n-k-j-1} \bigg( \sum_{i=0}^{n+k}  (-1)^{i}\binom{n+k}{i} x^{n+k-i} \\
& & \quad - \,  \epsilon(j,k,m,n) (-1)^{m - j - n} \binom{n + k}{m-n-j} x^{n+ k -m + j + n} \bigg) dx \bigg) \! \biggr) \bigg\rvert_{x=1}\\
& & \quad  + \,  \epsilon(j,k,m,n)  (-1)^{m + k + 1} \binom{n+k}{m -n -j}\\
&=& \frac{-1}{(j-1)! (n+k - j)!} \bigg(\! \Big(\frac{d}{dx}\Big)^{n+k} \biggl(x^{j-1}  \int x^{m-2n-k-j-1} \bigg( (x-1)^{n+k} \\
& & \quad - \,  \epsilon(j,k,m,n) (-1)^{m - j - n} \binom{n + k}{m-n-j} x^{2n+ k -m + j} \bigg) dx \bigg)  \! \biggr) \bigg\rvert_{x=1}\\
& & \quad  + \,  \epsilon(j,k,m,n)  (-1)^{m + k + 1} \binom{n+k}{m -n -j}\\
&=& \frac{-1}{(j-1)! (n+k - j)!} \bigg(\! \Big(\frac{d}{dx}\Big)^{n+k} \biggl(x^{j-1}  \int \bigg( x^{m-2n-k-j-1}  (x-1)^{n+k} \\
& & \quad - \,  \epsilon(j,k,m,n) (-1)^{m - j - n} \binom{n + k}{m-n-j} x^{-1}  \bigg) dx \bigg) \! \biggr) \bigg\rvert_{x=1}\\
& & \quad  + \,  \epsilon(j,k,m,n)  (-1)^{m + k + 1} \binom{n+k}{m -n -j}\\
&=& \frac{-1}{(j-1)! (n+k - j)!} \bigg( \! \Big(\frac{d}{dx}\Big)^{n+k} \biggl(x^{j-1}  \int x^{m-2n-k-j-1}  (x-1)^{n+k} dx \bigg)\!  \bigg)  \bigg\rvert_{x=1} \\
& & \quad +  \,   \frac{1}{(j-1)! (n+k - j)!} \epsilon(j,k,m,n) (-1)^{m - j - n} \binom{n + k}{m-n-j} \bigg( \! \Big( \frac{d}{dx}\Big)^{n + k}  x^{j-1} \ln x  \bigg)  \bigg\rvert_{x=1}\\
& & \quad  + \,  \epsilon(j,k,m,n)  (-1)^{m + k + 1} \binom{n+k}{m -n -j}\\
&=& \frac{-1}{(j-1)! (n+k - j)!} \bigg(\! \Big(\frac{d}{dx}\Big)^{n+k} \biggl(x^{j-1}  \int x^{m-2n-k-j-1}  (x-1)^{n+k} dx \bigg) \! \bigg)  \bigg\rvert_{x=1} \\
& & \quad +  \,  (-1)^m \epsilon(j,k,m,n)  \binom{n + k}{m-n-j} \bigg(   \frac{1}{(j-1)! (n+k - j)!}(-1)^{ - j - n}\\
& & \quad \cdot \bigg( \! \Big( \frac{d}{dx}\Big)^{n + k}  x^{j-1} \ln x \bigg)  \! \bigg) \bigg\rvert_{x=1}  +  (-1)^{ k + 1} \bigg) .
\end{eqnarray*} 

The proof is thus reduced to showing that for $n + k >0$, $j = 1, \dots, n+k$, and $m\geq n+ 1$,
\begin{equation}\label{first-part}
 \bigg( \! \Big(\frac{d}{dx}\Big)^{n+k}\bigg(x^{j-1}\int x^{m-2n-k-j-1}(-x+1)^{n+k}dx \bigg) \! \bigg)  \bigg\rvert_{x=1}=0,
\end{equation}
and 
\begin{equation}\label{second-part}
 \frac{1}{(j-1)! (n+k - j)!}(-1)^{ - j - n}\bigg( \! \Big( \frac{d}{dx}\Big)^{n + k}  x^{j-1} \ln x \bigg) \! \bigg) \bigg\rvert_{x=1} + (-1)^{k+1} = 0.
\end{equation}

Since the $(n+k)$th derivative of a product of functions, $f$ and $g$ is given by
\begin{equation}\label{derivprod}
\Big(\frac{d}{dx}\Big)^{n+k}(fg)=\sum_{i=0}^{n+k}\binom{n+k}{i}\biggl(\! \Big(\frac{d}{dx}\Big)^{n+k-i}f \biggr) \biggl(\! \Big(\frac{d}{dx}\Big)^ig \biggr),
\end{equation}
applying this in the case when  $f=x^{j-1}$ and $g=\int x^{m-2n-k-j-1}(-x+1)^{n+k}dx,$ we have
\begin{eqnarray*}
\lefteqn{\Big(\frac{d}{dx}\Big)^{n+k}\bigg(x^{j-1}\int x^{m-2n-k-j-1}(-x+1)^{n+k}dx \bigg)}\\
&=& \sum_{i=0}^{n+k}\binom{n+k}{i}\biggl(\! \Big(\frac{d}{dx}\Big)^{n+k-i}x^{j-1}\biggr) \biggl( \! \Big(\frac{d}{dx}\Big)^i \int x^{m-2n-k-j-1}(-x+1)^{n+k}dx\biggr)\\
&=& (j-1) (j-2) \cdots (j-n - k) x^{j - n - k-1} \int x^{m-2n-k-j-1}(-x+1)^{n+k}dx \\
& & \quad + \, \sum_{i=1}^{n+ k}\binom{n+k}{i}(j-1)(j-2)\cdots(j-n-k+i)x^{j-n - k +i-1}\\
& & \quad \cdot  \Big(\frac{d}{dx}\Big)^{i-1} (x^{m-2n-k-j-1}(-x+1)^{n+k} ).
\end{eqnarray*}
The first term in the equation above is zero since $j = 1, \dots, n+k$, and each of the terms in the summation are zero at $x = 1$ since they each contain a factor of $(1 - x)$.  Thus (\ref{first-part}) holds.    Applying (\ref{derivprod}) to the case when $f = x^{j-1}$ and $g = \ln x$, we have that
\begin{eqnarray*}
\lefteqn{\frac{1}{(j-1)! (n+k - j)!}(-1)^{ - j - n}\bigg( \! \Big( \frac{d}{dx}\Big)^{n + k}  x^{j-1} \ln x \bigg)  \! \bigg) \bigg\rvert_{x=1} \! \! \! + (-1)^{k+1}}\\
&=&\frac{1}{(j-1)! (n+k - j)!}(-1)^{ - j - n}  \sum_{i=0}^{n+k}\binom{n+k}{i}\biggl(\! \Big(\frac{d}{dx}\Big)^{n+k-i}x^{j-1}\biggr) \biggl( \! \Big(\frac{d}{dx}\Big)^i \ln x \biggr)  \bigg\rvert_{x=1} \!  \! \! \! + (-1)^{k+1}\\
&=&   \frac{1}{(j-1)! (n+k - j)!}(-1)^{ - j - n}  \! \!  \sum_{i=n+k - j + 1}^{n+k}\binom{n+k}{i}\biggl(\! \Big(\frac{d}{dx}\Big)^{n+k-i}x^{j-1}\biggr) \biggl( \! \Big(\frac{d}{dx}\Big)^i \ln x \biggr)\bigg\rvert_{x=1} \\
& & \quad  + \, (-1)^{k+1} \\
&=&\frac{1}{(j-1)! (n+k - j)!}(-1)^{ - j - n} \bigg( \sum_{i=n+k -j +1}^{n+ k}\binom{n+k}{i}(j-1)(j-2)\cdots(j-n-k+i)\\
& & \quad \cdot x^{j-n - k +i-1} (-1) (-2) \cdots (-i + 1) x^{-i} \bigg) \bigg\rvert_{x=1} \! \! \! + (-1)^{k+1}\\
&=&\frac{1}{(j-1)! (n+k - j)!}(-1)^{ - j - n} \bigg( \sum_{i=n+k -j +1}^{n+ k}\binom{n+k}{i}(j-1)(j-2)\cdots(j-n-k+i)
\\
& & \quad (-1) (-2) \cdots (-i + 1)  \bigg)  + (-1)^{k+1} \\
&=&\frac{1}{(j-1)! (n+k - j)!}(-1)^{ - j - n} \bigg( \sum_{i=n+k -j +1}^{n+ k}\binom{n+k}{n+k - i} \frac{(j-1)!}{(j - n - k + i -1)!}\\
& & \quad (-1)^{i-1}  (i-1)!  \bigg)  + (-1)^{k+1} \\
&=&\frac{1}{(j-1)! (n+k - j)!}(-1)^{ - j - n} \bigg( \sum_{r=0}^{j-1}\binom{n+k}{j-1-r} \frac{(j-1)!}{r!}  (-1)^{n+k - j + r}  (n+ k - j + r)!  \bigg)  \\
& & \quad + \, (-1)^{k+1} \\
&=& \frac{1}{(n+k - j)!}(-1)^{ k} \bigg( \sum_{r=0}^{j-1} (-1)^{ r}  \binom{n+k}{j-1-r} \frac{(n+ k - j + r)! }{r!}  \bigg)  + (-1)^{k+1}  \\
&=&(-1)^k  \sum_{r=0}^{j-1} (-1)^{ r}  \binom{n+k}{j-1-r} \binom{n+ k - j + r}{r}  + (-1)^{k+1}  .
\end{eqnarray*}

We are done if we can show that 
\begin{equation}\label{last-equation}
\sum_{r=0}^{j-1}(-1)^r\binom{n+k}{j-1-r}\binom{n+k-j+r}{r}=1.
\end{equation}
Noting that $\binom{p}{q} = 0$ if $q<0$, then substituting $l = j - 1 - r$, then using the fact that $\binom{p}{q} = \binom{p}{p-q}$, then using the fact that $\binom{p}{q} = 0$ if $p\geq 0$ and $q>p$ to extend the range of $l$,  and then using Lemma \ref{appendix-lemma}, Eqn. (\ref{combo-lemma-id2}), we have that 
\begin{eqnarray*}
\lefteqn{\sum_{r=0}^{j-1}(-1)^r\binom{n+k}{j-1-r}\binom{n+k-j+r}{r}}\\
&=& \sum_{r=-1}^{j-1}(-1)^r\binom{n+k}{j-1-r}\binom{n+k-j+r}{r} \\
&=&\sum_{l=0}^{ j }(-1)^{j-1+l}\binom{n+k}{l}\binom{n+k-l-1}{j-l-1}\\
&=&\sum_{l=0}^{  j}  (-1)^{j-1+l}\binom{n+k}{l}\binom{n+k-l-1}{n+k-j}\\
&=&\sum_{l \in \mathbb{N} }(-1)^{j-1+l}\binom{n+k}{l}\binom{n+k-l-1}{n+k-j}\\
&=&(-1)^{j-1}\sum_{ l \in \mathbb{N} } (-1)^{l}\binom{n+k}{l}\binom{n+k-l-1}{n+k-j}\\
&=&(-1)^{j-1}\binom{-1}{j-1}=1,\\
\end{eqnarray*}
proving (\ref{last-equation}).
This then concludes the proof of \eqref{second-part}, and thus the Proposition.
\end{proof}

\end{document}